\documentclass[11pt]{amsart}
\usepackage[utf8]{inputenc}
\pdfoutput = 1

\usepackage{amsmath, amssymb, amsthm, mathrsfs}
\usepackage[margin=1in]{geometry}
\usepackage[english]{babel}
\usepackage{xcolor}
\usepackage{graphicx}
\usepackage{tikz}
\usepackage{enumitem}
\usepackage{hyperref}

\newtheorem{theorem}{Theorem}[section]

\newtheorem{lemma}[theorem]{Lemma}
\newtheorem{prop}[theorem]{Proposition}

\theoremstyle{definition}
\newtheorem{definition}[theorem]{Definition}
\theoremstyle{remark}
\newtheorem*{remark}{Remark}
\numberwithin{equation}{section}

\let\cal\mathcal

\let\bb\mathbb

\newcommand{\DF}{\mathcal{E}}

\def\octagon[#1](#2:#3)(#4){%
  \draw [#1] (#2 + 0.70711*#4, #3) -- (#2 + 1.70711*#4, #3) -- (#2 + 2.41421*#4, #3 + 0.70711*#4) -- (#2 + 2.41421*#4, #3 + 1.70711*#4) -- (#2 + 1.70711*#4, #3 + 2.41421*#4) -- (#2 + 0.70711*#4, #3 + 2.41421*#4) -- (#2, #3 + 1.70711*#4) -- (#2, #3 + 0.70711*#4) -- (#2 + 0.70711*#4, #3); }

\newcommand{\radvertex}[2]{\draw [fill] (#1) circle [radius = #2];}

\title{Resistance Scaling on $4N$-Carpets }
\author[Canner]{Claire Canner}
\address{Claire Canner\\Rochester Institute of Technology}
\email{clairemcanner@gmail.com}

\author[Hayes]{Christopher Hayes}
\address{Christopher Hayes\\Department of Mathematics\\University of Connecticut\\Storrs, CT 06269-1009\\U.S.A.}
\email{christopher.k.hayes@uconn.edu}

\author{Shinyu Huang}
\address{Shinyu Huang\\Williams College}
\email{wsh1@williams.edu}

\author{Michael Orwin}
\address{Michael Orwin\\ Kalamazoo College}
\email{orwinmc@gmail.com}

\author{Luke~G. Rogers.}
\address{Luke~G. Rogers\\Department of Mathematics\\University of Connecticut\\Storrs, CT 06269-1009\\U.S.A.}
\email{luke.rogers@uconn.edu}

\thanks{Work supported by NSF DMS REU 1659643}

\begin{document}
\begin{abstract}
The $4N$ carpets are a class of infinitely ramified self-similar fractals with a large group of symmetries.   For a $4N$-carpet $F$, let $\{F_n\}_{n \geq 0}$ be the natural decreasing sequence of compact pre-fractal approximations with $\cap_nF_n=F$. On each $F_n$, let $\cal E(u, v) = \int_{F_N} \nabla u \cdot \nabla v \, dx$ be the classical Dirichlet form and $u_n$ be the unique harmonic function on $F_n$ satisfying a mixed boundary value problem corresponding to assigning a constant potential between two specific subsets of the boundary.  Using a method introduced by Barlow and Bass~\cite{BB1}, we prove a resistance estimate of the following form: there is $\rho=\rho(N) > 1$ such that $\cal E(u_n, u_n)\rho^{n}$ is bounded above and below by constants independent of $n$. Such estimates have implications for the existence and scaling properties of Brownian motion on $F$.
\end{abstract}

\maketitle

\section{Introduction}

The $4N$ carpets are  a class of self-similar fractals related to the classical Sierpi\'nski Carpets.  They are defined by a finite set of similitudes with a single contraction ratio, are highly symmetric, and are post-critically infinite.  Two examples, the octacarpet ($N=2$) and dodecacarpet ($N=3$) are shown in Figure~\ref{fig:oct-dodeca}. We do not consider the case $N=1$ which is simply a square.  

\begin{figure}
\centering
  \includegraphics[width=0.3\linewidth]{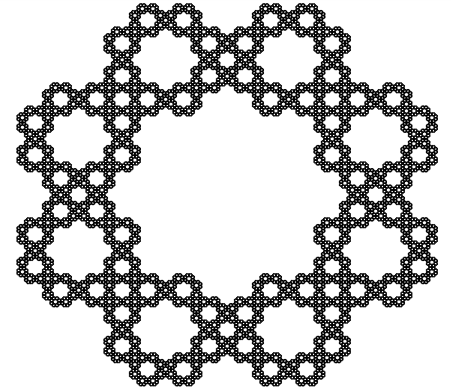} \ \ \ \ \ \    \includegraphics[width=0.3\linewidth]{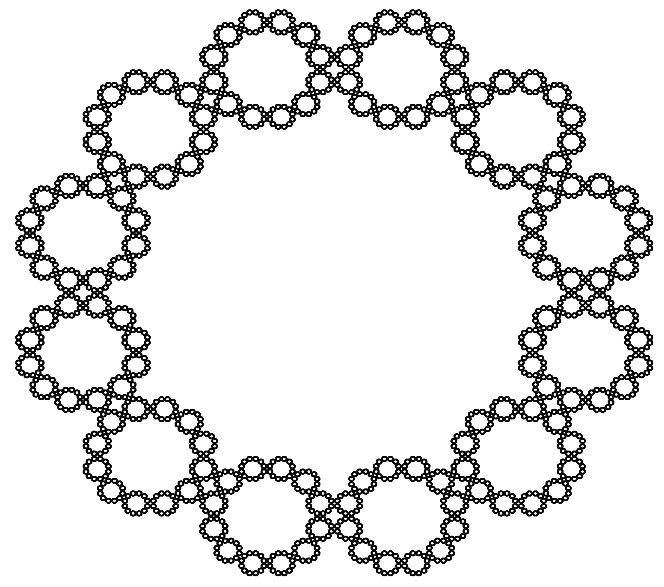}
\caption{The octacarpet $(N=2)$ and dodecacarpet $(N=3)$ are $4N$-Carpets.}
\label{fig:oct-dodeca}
\end{figure}

The construction of $4N$ carpets is as follows; illustrations for $N=2$ are in Figure~\ref{fig:octconstr}.  Fix $N\geq2$, let $\Lambda(N)=\{0,\dotsc, 4N-1\}$ and $C_j(N)=\exp\frac{(2j-1)i\pi}{4N}\in\mathbb{C}$. Let $F_0$ denote the convex hull of $\{C_j(N),j\in\Lambda(N)\}$.  Consider contractions $\phi_j(x) =r(x-C_j)+C_j$ where the ratio $r=r(N)=(1+\cot(\pi/4N))^{-1}$ is chosen so $\phi_j(F_0)\cap\phi_k(F_0)$ is a line segment.  For a set $A$ define $\Phi(A)=\cup_{j=0}^{4N-1}\phi_j(A)$ and let $\Phi^n$ denote the $n$-fold composition. $\Phi$ is a contraction on the space of non-empty compact sets in $\bb C$ with the Hausdorff metric (\cite{Kigami}, pg. 11). Then let  $F_n=\Phi^n(F_0)$ and $F=\cap_n F_n$ be the unique non-empty compact set such that $\Phi(F)=F$ (see~\cite{Hutch81}).  We call $F$ the $4N$-carpet. Since one may verify the Moran open set condition is valid for the interior of $F_0$, \cite[Theorem~5.3(2)]{Hutch81} implies its Hausdorff dimension is $d_f=-\log4N/\log r(N)=\log 4N/\log (1+\cot(\pi/4N))$.

\begin{figure}
\centering
\includegraphics[width=.3\linewidth]{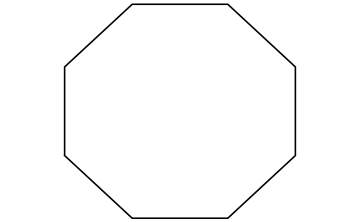} \ \ 
\includegraphics[width=.3\linewidth]{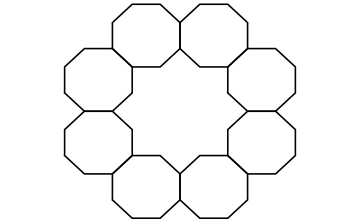} \ \ 
 \includegraphics[width=.3\linewidth]{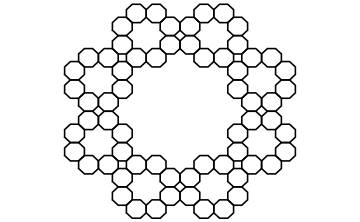}
\caption{The pre-carpets $F_0$, $F_1$ and $F_2$ for the octacarpet ($N=2$).}\label{fig:octconstr}
\end{figure}

This paper is concerned with a physically-motivated problem connected to the resistance of the $4N$ carpet, and is closely related to well-known results of Barlow and Bass~\cite{BB1}.  To state it we need some further notation.  Writing subindices modulo $4N$, let $L_j$ be the line segment from $C_j$ to $C_{j+1}$ (these are shown  for the case $N=2$ in the left diagram in Figure~\ref{fig:currentconstruction}).  Then define
\begin{align}\label{eqn:AnBn}
	 A_n=F_n\cap\Bigl(\cup_{k=0}^{N-1} L_{4k}\Bigr) && B_n=F_n\cap\Bigl(\cup_{k=0}^{N-1} L_{4k+2}\Bigr)
 	\end{align}
These sets are shown for the case of the dodecacarpet ($N=3$) in Figure~\ref{fig:AmBm}.

Supposing $F_n$ to be constructed from a thin, electrically conductive sheet let $R_n=R_n(N)$ be the effective resistance (as defined in~\eqref{eq:domainresist} below) when the edges of $A_n$ are short-circuited at potential $0$ and those of $B_n$ are short-circuited at potential $1$.

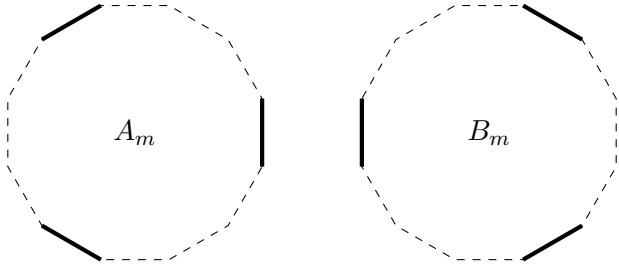
\begin{figure}
	\centering
	\begin{tikzpicture}[scale=1.75, rotate = 330]
	
	\draw[dashed]  (0.9659258, 0.25881904)--(0.70710677, 0.70710677)--(0.25881904, 0.9659258)--(-0.25881904, 0.9659258)--(-0.70710677, 0.70710677)--(-0.9659258, 0.25881904)--(-0.9659258, -0.25881904)--(-0.70710677, -0.70710677)--(-0.25881904, -0.9659258)--(0.25881904, -0.9659258)--(0.70710677, -0.70710677)--(0.9659258, -0.25881904)--(0.9659258, 0.25881904);
	
	\draw [ultra thick] (0.9659258, 0.25881904)--(0.70710677, 0.70710677);
	\draw [ultra thick] (-0.70710677, 0.70710677)--(-0.9659258, 0.25881904); 
	\draw [ultra thick] (-0.25881904, -0.9659258)--(0.25881904, -0.9659258);
	\node at (0, 0) {$A_m$};	
\end{tikzpicture}\hspace{0.5in}\begin{tikzpicture}[scale=1.75, rotate = 150]

\draw[dashed]  (0.9659258, 0.25881904)--(0.70710677, 0.70710677)--(0.25881904, 0.9659258)--(-0.25881904, 0.9659258)--(-0.70710677, 0.70710677)--(-0.9659258, 0.25881904)--(-0.9659258, -0.25881904)--(-0.70710677, -0.70710677)--(-0.25881904, -0.9659258)--(0.25881904, -0.9659258)--(0.70710677, -0.70710677)--(0.9659258, -0.25881904)--(0.9659258, 0.25881904);

\draw [ultra thick] (0.9659258, 0.25881904)--(0.70710677, 0.70710677);
\draw [ultra thick] (-0.70710677, 0.70710677)--(-0.9659258, 0.25881904); 
\draw [ultra thick] (-0.25881904, -0.9659258)--(0.25881904, -0.9659258);
\node at (0, 0) {$B_m$};	
\end{tikzpicture}
	
	\caption{The thick lines indicate $A_m$ and $B_m$ for the dodecacarpet ($N=3)$). }\label{fig:AmBm}
\end{figure}
 
Bounds for $R_n$ have a well-known connection to crossing time estimates for Brownian Motion (see, for example, \cite[Theorem~2.7]{BarlowDiffusionsonFractals}). In the case of the Sierpi\'nski Carpet such estimates play a significant role in the Barlow-Bass approach to establishing properties of the Brownian motion constructed in~\cite{BBExistence}. Specifically, these estimates are used to establish the behavior of the resulting heat kernel $p_t(x,y)$ under space rescaling and hence prove existence of the limit $\lim_{t\to0^+} \frac{\ln p_t(x,x)}{\ln t}$, which is used to define the spectral dimension. See~\cite{BB1,BBTransDensities} for details and~\cite{BBS} for numerical estimates of the spectral dimension via estimation of the resistance scaling. 

There have been considerable developments regarding the Dirichlet form on the  Sierpi\'nski carpet, among which we note the proof of uniqueness~\cite{BBKT}, more general results on the geometry of the spectral dimension~\cite{KajinoSA}, and a proof of existence of the form by a non-probabilistic method~\cite{GrigYang}. The original results of Barlow and Bass~\cite{BB1} on scaling of the resistance for the Sierpi\'nski carpet were extended to more general types of carpets in~\cite{BB99,McGillivray}, and further improved using a different approach in~\cite{KajinoElementaryWD}.  Resistance estimates for the Strichartz hexacarpet are in~\cite{KPBT19}.
 
  With regard to the $4N$ carpets considered here, there are few results in the literature.  For the case $N=2$, which is called either the octacarpet or octagasket, there are some results regarding features one might expect the spectrum of a Laplacian to have, provided that one exists: \cite{StrichartzOuterApprox} contains numerical data and results from Strichartz's ``outer approximation'' method, and~\cite{StrichartzPeanoCurve} has results from a method involving approximation of the set by a Peano curve.    
The results most closely connected to the present work are in the PhD thesis of Ulysses Andrews~\cite{Andrews}, where the Barlow-Bass method is used to prove the existence of a local regular Dirichlet form on $4N$ carpets under several assumptions, one of which is a resistance estimate that follows readily from Theorem~\ref{thm:main} below.
 
Our main result is the following theorem.

\begin{theorem}\label{thm:main}
For fixed $N\geq2$ there is a constant $\rho=\rho(N)$ such that: 

$$\frac9{44N} R_0\rho^n \leq R_n \leq \frac{44N}9 R_0\rho^n$$

\end{theorem}
\begin{proof}
The majority of the work is to establish (in Theorem~\ref{thm:mainest} below) that there are constants $c,C$ so $cR_nR_m\leq R_{n+m}\leq C R_nR_m$. Then $S_n=\log cR_n$ is superadditive and $S_n'=\log CR_n$ is subadditive, so Fekete's lemma implies $\lim_{n\rightarrow\infty}\frac1n S_n =  \sup_n\frac1nS_n$ and $\lim_{n\rightarrow\infty}\frac1nS_n'=\inf_n\frac1nS_n'$.  However $\lim_{n\rightarrow\infty}\frac1n(S_n-S_n') = 0$, so defining  $\log\rho$ to be the common limit we conclude $\frac1nS_n\leq \log\rho\leq\frac1nS_n'$ and thus $cR_n\leq \rho^n\leq CR_n$.  

%The constants $c$ and $C$ in the stated bound are those from Theorem~\ref{thm:mainest}.
\end{proof}

\section{Resistance, flows and currents}
We recall some necessary notions regarding Dirichlet forms on graphs and on Lipschitz domains. Our treatment generally  follows~\cite{BB1}, which in turn refers to~\cite{DoyleSnell} for the graph case.

\subsection{Graphs}
On a finite set of points $G$ suppose we have $g:G\times G\to\mathbb{R}$ satisfying for all $x,y\in G$ that $g(x,y)=g(y,x)$, $g(x,y)\geq0$ and $g(x,x)=0$. We call $g$ a conductance. It defines a Dirichlet form by
\begin{equation*}
	\DF_G(u,u) = \frac12 \sum_{x\in G}\sum_{y\in G} g(x,y) \bigl( u(x)-u(y)\bigr)^2.
	\end{equation*}
For disjoint subsets $A$, $B$ from $G$ the effective resistance between them is $R_G(A,B)$ defined by
\begin{equation}\label{eq:R_G}
	R_G(A,B)^{-1}=\inf \{ \DF_G(u,u): u|_A=0,\ u|_B=1\}.
	\end{equation}
The set of functions in \eqref{eq:R_G} are called feasible potentials. Viewing $G$ as the vertex set of a graph with an edge from $x$ to $y$ when $g(x, y) > 0$ we note that if $G$ is connected then the infimum is attained at a unique potential $\tilde{u}_G$.

A current from $A$ to $B$ is a function $I$ on the edges of the conductance graph, meaning $I:\{x,y:g(x,y)>0\}\to \mathbb{R}$,
with properties: $I(x,y)=-I(y,x)$ for all $x,y$ and $\sum_{y\in G} I(x,y)=0$ if $x\not\in A\cup B$.  It is called a feasible current if it has unit flux, meaning:
$$\sum_{x\in B}\sum_{y\in G} I(x,y) = - \sum_{x\in A}\sum_{y\in G} I(x,y) = 1$$

Note that the first equality is a consequence of the definition of a current. The energy of the current is defined by
\begin{equation*}
	E_G(I,I) =\frac12\sum_{x\in G}\sum_{x\in G} g(x,y)^{-1} I(x,y)^2.
	\end{equation*}

\begin{theorem}[\protect{\cite[Section~1.3.5]{DoyleSnell}}]\label{thm:ThomsponPrinciple}
$$R_G(A,B) = \inf \{E_G(I,I): I\text{ is a feasible current}\}.$$
\end{theorem}
This well-known result, often called Thomson's Principle, is proven by showing that for the optimal potential $\tilde{u}_G$ one may define a current by $\nabla u_G(x,y)=(\tilde{u}_G(y)-\tilde{u}_G(x))g(x,y)$, this current has flux $R_G(A,B)^{-1}$ and the optimal current which attains the infimum in the theorem is $\tilde{I}_G=R_G(A,B)\nabla \tilde{u}_G$.  We note that the theorem in the reference is only for the case where $A$ and $B$ are singleton sets, but that the argument requires minimal changes to cover the more general case, and indeed the general case is often treated by ``shorting'' each set to a point.

\subsection{Lipschitz domains}\label{ssec:Lipdom}

We shall need corresponding results on each of our prefractal sets $F_n$, for the sets $A_n$ and $B_n$ defined in~\eqref{eqn:AnBn}.  To this end we must  define a space of potentials and of currents for which suitable integrability criteria are valid; several ways to do this are possible, and it can sometimes be difficult to check all details for the approaches in the literature.

Let $\Omega\subset\mathbb{C}$ be a Lipschitz domain, suppose $A$, $B$ are disjoint closed subsets of $\partial\Omega$ and write $\sigma$ for the surface measure on $\partial\Omega$ and $\nu$ for the interior unit normal. Let $H^1(\Omega)$ denote the Sobolev space with one derivative in $L^2$. Note that since $\Omega$ is a Sobolev extension domain (see \cite[Chapter~6]{Steindiffprops}) the space $C^1(\bar{\Omega})$ is dense in $H^1(\Omega)$.   Also, the trace of $H^1(\Omega)$ to the Lipschitz boundary is the fractional Sobolev space $H^{1/2}(\partial\Omega,d\sigma)$ and $H^1(\Omega)$ convergence implies convergence in $H^{1/2}(\partial\Omega)$ (see, for example~\cite[Theorem~1 of Chapter VII]{JonssonWallin}).  We then define a feasible potential for the pair $(A,B)$ to be $u\in H^1(\Omega)$ which satisfies $u|_A=0$ and $u|_B=1$ in the sense of $H^{1/2}$. It is easy to check that feasible potentials exist, and we define the effective resistance from $A$ to $B$ by
\begin{equation}\label{eq:domainresist}
	R_\Omega(A,B)^{-1} = \inf \{\DF_\Omega(u,u):\ \text{$u$ is a feasible potential.} \}
	\end{equation}

Our space of currents is defined to be the subspace of $L^2(\Omega,\mathbb{R}^n)$ functions with vanishing weak divergence. We need a Gauss-Green theorem to determine a sense in which the boundary values exist; this is standard (even in greater generality) but included for the convenience of the reader.

\begin{lemma}[A Gauss-Green theorem]\label{thm:GaussGreen}
For $u\in H^2(\Omega)$ and $I\in L^2(\Omega,\mathbb{R}^n)$ with $\nabla I=0$ in the weak sense,
\begin{equation*}
\int_\Omega (\nabla u)\cdot I = -\int_{\partial\Omega} u I\cdot d\nu,
\end{equation*}
where the boundary values $I\cdot d\nu$ exist as an element of $H^{-1/2}(\partial\Omega)$, the dual of $H^{1/2}(\partial\Omega)$.
\end{lemma}
\begin{proof}
Using a classical version of the Gauss-Green theorem (see~\cite[Theorem~1 of Section~5.8]{EvansGariepy} for a proof applicable to the situation of a Lipschitz boundary)  for $u'\in C^1(\bar{\Omega})$ and $I'\in C^1(\bar{\Omega},\mathbb{R}^n)$
\begin{equation}\label{eq:GGstep}
-\int_{\partial\Omega} u' I'\cdot d\nu
=\int_\Omega \nabla\cdot(u' I')
=\int_\Omega (\nabla u')\cdot I' +\int_\Omega u' \nabla\cdot I'
\end{equation}

We first observe that in~\eqref{eq:GGstep} we may  approximate $u\in H^2(\Omega)$ by $u'\in C^1(\bar{\Omega})$ in $H^2(\Omega)$ norm. On the left we get convergence of the boundary values in $H^{1/2}$, and on the right we have convergence of $\nabla u'$ to $\nabla u$ in $L^2(\Omega,\mathbb{R}^n)$ and of $u'$ to $u$ in $L^2(\Omega)$.

Now the vanishing weak divergence of $I$ determines, in particular, that $\nabla\cdot I\in L^2(\Omega)$. We need the standard but non-trivial fact that one can approximate by $I'$ so that $\|I-I'\|_{L^2}+\|\nabla \cdot I-\nabla\cdot I'\|_{L^2}<\epsilon$. From this it is apparent that  the right side of~\eqref{eq:GGstep} converges to $\int (\nabla u)\cdot I$.

Since the right side converges, the left must also. The limit of $u$ is in $H^{1/2}(\partial\Omega)$, so $I\cdot d\nu$ exists as an element of the dual $H^{-1/2}(\partial\Omega)$.
\end{proof}

We may now define a feasible current for the pair $(A,B)$ as $I\in L^2(\Omega)$ with $\nabla\cdot I=0$ and  $I_{\partial\Omega\setminus(A\cup B)}=0$ as an element of $H^{-1/2}$, and the flux integrals
\begin{equation}\label{eq:fluxintegralseqopp}
\int_B J\cdot \nu d\sigma= - \int_A J\cdot \nu d\sigma =1
\end{equation}
where we note that the equality~\eqref{eq:fluxintegralseqopp} follows from Lemma~\ref{thm:GaussGreen}.

Our work here depends crucially on the following result, which is a special case of~\cite[Theorem~2.1]{RBrown}.  The reader may recognize that Lax-Milgram provides a solution to the stated Dirichlet problem, so we emphasize that the main content of the theorem is that the boundary gradient is in $L^2(\partial\Omega)$.  It should be noted that this result is not valid for arbitrary mixed boundary value problems on Lipschitz domains; in particular, in~\cite{RBrown} it is required that the pieces of the boundary on which the Dirichlet and Neumann conditions hold meet at an angle less than $\pi$. This condition is true for the sets $A_n,B_n,F_n$.  
Finally, the reader may observe that since our domains are polygonal we could have obtained the desired result by classical techniques such as those of Grisvard~\cite[Section~4.3.1]{Grisvard}.  

\begin{theorem}\label{thm:BVPsoln}
For the choices of domain $\Omega$ and sets $A$, $B$ considered in this paper, there is a unique $\tilde{u}_\Omega\in H^1(\Omega)$ with $\nabla \tilde{u}_\Omega\in L^2(d\sigma)$ which solves the mixed boundary value problem
\begin{equation*}
	\begin{cases} \Delta \tilde{u}_\Omega=0 \text{ in $\Omega$}\\
	\tilde{u}_\Omega|_A =0,\ \tilde{u}_\Omega|_B=1 \\
	\frac{\partial \tilde{u}_\Omega}{\partial \nu} =0 \text{ a.e. } d\sigma \text{ on $\partial\Omega\setminus(A\cup B)$.}
	\end{cases}
	\end{equation*}
\end{theorem}

Once this is known, the proofs of~\cite[Proposition~2.2 and Theorem~2.3]{BB1} may be duplicated in our setting, using Lemma~\ref{thm:GaussGreen} above in place of~\cite[Lemma~2.1]{BB1}, to prove the following analogue of the previously stated result for graphs.  It is perhaps worth remarking that the argument uses that $\nabla\tilde{u}_\Omega$ is a current, as this explains why we need the boundary gradient to be in $L^2(\partial\Omega,d\sigma)$ (or at least in $H^{-1/2}(\partial\Omega)$) as well as illustrating the connection between the Dirichlet problem and the requirement that currents have vanishing divergence. We also note that standard results about harmonic functions ensure $\tilde{u}_\Omega$ has a representative that is continuous at points of $A$; this will be relevant later.

\begin{theorem}\label{thm:resistfromcurrent}
The function $\tilde{u}_\Omega$ from Theorem~\ref{thm:BVPsoln} is the unique minimizer of~\eqref{eq:domainresist} so satisfies $R_\Omega(A,B)^{-1}=\DF_\Omega(\tilde{u}_\Omega,\tilde{u}_\Omega)$.  Moreover $\tilde{J}_\Omega=R_\Omega(A,B)\nabla \tilde{u}_\Omega$ is the unique minimizer for\/
 $\inf\{E_\Omega(J,J): J \text{ is a feasible current}\}$ and thus $E_\Omega(\tilde{J}_\Omega,\tilde{J}_\Omega)=R_\Omega(A,B)$.
\end{theorem}

We close this section with an observation that will be used to glue potentials and currents in the construction in Section~\ref{sec:resistests}.

\begin{lemma}\label{lem:potandcurglue}
Suppose $\Omega\subset\mathbb{C}$ is a Lipschitz domain symmetric under reflection in a line $L$ and write $\Omega_{\pm}$ for the intersection with the half planes on either side of $L$.
\begin{itemize}
\item[(i)] Let $u_{\pm}\in H^1(\Omega_\pm)$ respectively.  Setting $u=u_\pm$ on $\Omega_\pm$ defines a function in  $H^1(\Omega)$ if and only if $u_\pm$ are equal a.e.\ on $\Omega\cap L$. 
\item[(ii)] Let $I_\pm\in L^2(\Omega_{\pm})$ satisfy $\nabla\cdot I_{\pm}=0$ on $\Omega_\pm$ respectively. Then $I=I_\pm$ on $\Omega_{\pm}$ satisfies $\nabla\cdot I=0$ on $\Omega$ if and only if $\nabla\cdot I_+(z)=-\nabla\cdot I_-(z)$ in the weak sense on $\Omega\cap\mathbb{R}$.
\end{itemize}
\end{lemma}
\begin{proof}
We first rotate and translate so that $L=\mathbb{R}$, and $\Omega_\pm$ are the intersections with the upper and lower half planes. Evidently all of the relevant notions are invariant under this Euclidean motion.

For (i) we use the standard characterization that if $u\in L^2$ then $u\in H^1$ if and only if it has a representative that is absolutely continuous on almost all line segments in the domain that are parallel to the axes, and the classical derivatives on these line segments are in $L^2$ (see~\cite[Theorem~2.1.4]{Ziemer}).  Observe that all such segments in either of $\Omega_\pm$ are in $\Omega$, and the only segments in $\Omega$ that are not in one of $\Omega_\pm$ are those that are perpendicular to and cross the real axis. Then the absolute continuity with $L^2$ derivatives on each line segment is valid if and only if the representatives from $\Omega_\pm$ have the same value at the intersection of the line segment with the real axis.

For (ii) it is apparent that $I\in L^2(\Omega)$, so the relevant question is whether $\nabla\cdot I=0$. Certainly $\nabla\cdot I=0$ on $\Omega$ implies the same on $\Omega_\pm$.  For the converse, suppose $\nabla\cdot I=0$ on $\Omega_\pm$, so $\int \nabla f_\pm\cdot I=0$ for any $f_\pm\in C^1_0(\Omega_\pm)$.  We have $\nabla\cdot I=0$ if and only if for all $f\in C^1_0(\Omega)$
\begin{equation}\label{eqn:cutoff}
0=\int_\Omega I(z)\cdot \nabla f(z) 
= \int_{\Omega_+} I_+(z) \cdot \nabla f(z)+  \int_{\Omega_-} I_-(z) \cdot \nabla f(z).
\end{equation}
but if $g$ is a $C^1$ cutoff in a small neighborhood of $\mathbb{R}$ then $\int_\Omega I\cdot \nabla ((1-g)f)=0$ because $(1-g)f$ is a sum of  functions in $C^1_0(\Omega_\pm)$, so we only need~\eqref{eqn:cutoff} for $f$ supported in an arbitrarily small neighborhood of $\mathbb{R}$ and hence the condition is equivalent to the equality $\nabla\cdot I_+=-\nabla\cdot I_-$ in the weak sense on $\Omega\cap\mathbb{R}$.
\end{proof}

\section{Resistance Estimates}\label{sec:resistests}
Suppose Theorem~\ref{thm:BVPsoln} is applicable to the Lipschitz domain $\Omega$ and disjoint subsets $A,B \subset\partial\Omega$.  In light of~\eqref{eq:domainresist} we can bound the resistance from below by $\DF_\Omega(u,u)^{-1}$ for $u$ a feasible potential, and by the characterization in Theorem~\ref{thm:resistfromcurrent} we can bound the resistance from above by $E_\Omega(J,J)$ for $J$ a feasible current. To get good resistance estimates one must ensure the potential and current give comparable bounds.

We do this for the pre-carpet sets $F_{m+n}$ following the method of Barlow and Bass in~\cite{BB1}.  First we define graphs $G_m$ and $D_m$, which correspond to scale~$m$ approximations of a current and potential (respectively) on $F_m$, and for which the resistances are comparable.  Next we establish the key technical step, which involves using the symmetries of the $4N$ gasket  to construct a current with prescribed fluxes through certain sides $L_j\cap F_n$ from the optimal current on $F_n$ (see Proposition~\ref{prop:Fncurr}), and to construct a potential with prescribed  data at the endpoints of these sides from the optimal potential on $F_n$  (see Proposition~\ref{prop:Fnpot}).  Combining these results we establish the resistance bounds in Theorem~\ref{thm:mainest} by showing that the optimal current on $G_m$ can be used to define a  current on $F_{m+n}$ with comparable energy, and the optimal potential on $D_m$ can be used to define a potential on $F_{m+n}$ with comparable energy.

The two symmetries of $F$ and the pre-carpets $F_n$ that play an essential role are the rotation $\theta(z)=ze^{i\pi/2N}$ and complex conjugation.  They preserve $F$ and all $F_n$.

\subsection{Graph approximations}\label{ssec:graphs}

\begin{figure}
	\centering
	\begin{tikzpicture}[scale=2]
		\draw (0.9238795242202148, 0.3826834305423125)--(0.7653668629074022, 0.5411960918551251)--(0.5411961019689742, 0.5411960918551251)--(0.38268344065616167, 0.3826834305423125)--(0.38268344065616167, 0.15851266960388447)--(0.5411961019689742, 8.291071940114136e-09)--(0.7653668629074022, 8.291071940114136e-09)--(0.9238795242202148, 0.15851266960388447)--(0.9238795242202148, 0.3826834305423125);
		\draw (0.3826834305423126, 0.9238795242202148)--(0.5411960918551251, 0.7653668629074022)--(0.5411960918551251, 0.5411961019689742)--(0.3826834305423125, 0.38268344065616167)--(0.15851266960388452, 0.38268344065616167)--(8.291071995625288e-09, 0.5411961019689742)--(8.291071995625288e-09, 0.7653668629074022)--(0.15851266960388452, 0.9238795242202148)--(0.3826834305423126, 0.9238795242202148);
		\draw (-0.38268343054231246, 0.9238795242202148)--(-0.541196091855125, 0.7653668629074022)--(-0.541196091855125, 0.5411961019689742)--(-0.38268343054231246, 0.38268344065616167)--(-0.15851266960388447, 0.38268344065616167)--(-8.291071940114136e-09, 0.5411961019689742)--(-8.291071940114136e-09, 0.7653668629074022)--(-0.1585126696038844, 0.9238795242202148)--(-0.38268343054231246, 0.9238795242202148);
		\draw(-0.9238795242202148, 0.38268343054231263)--(-0.7653668629074022, 0.5411960918551251)--(-0.5411961019689742, 0.5411960918551251)--(-0.38268344065616167, 0.3826834305423126)--(-0.38268344065616167, 0.15851266960388452)--(-0.5411961019689743, 8.291072051136439e-09)--(-0.7653668629074023, 8.291072051136439e-09)--(-0.9238795242202148, 0.15851266960388458)--(-0.9238795242202148, 0.38268343054231263);
		\draw(-0.9238795242202148, -0.3826834305423124)--(-0.7653668629074024, -0.541196091855125)--(-0.5411961019689744, -0.541196091855125)--(-0.3826834406561618, -0.38268343054231246)--(-0.3826834406561618, -0.1585126696038844)--(-0.5411961019689743, -8.291071884602985e-09)--(-0.7653668629074023, -8.291071884602985e-09)--(-0.9238795242202148, -0.15851266960388435)--(-0.9238795242202148, -0.3826834305423124);
		\draw(-0.3826834277913151, -0.9238795253597152)--(-0.15851267235488267, -0.9238795253597152)--(-7.151571779218102e-09, -0.7653668601564043)--(-7.151571779218102e-09, -0.5411961047199718)--(-0.15851267235488267, -0.38268343951666095)--(-0.38268342779131514, -0.38268343951666106)--(-0.541196092994626, -0.5411961047199718)--(-0.541196092994626, -0.7653668601564043)--(-0.3826834277913151, -0.9238795253597152);
		\draw(0.3826834277913149, -0.9238795253597153)--(0.15851267235488245, -0.9238795253597153)--(7.151571557173497e-09, -0.7653668601564044)--(7.151571501662346e-09, -0.541196104719972)--(0.1585126723548824, -0.38268343951666106)--(0.38268342779131487, -0.38268343951666106)--(0.5411960929946258, -0.541196104719972)--(0.5411960929946258, -0.7653668601564043)--(0.3826834277913149, -0.9238795253597153);
		\draw (0.9238795253597152, -0.3826834277913152)--(0.9238795253597152, -0.15851267235488278)--(0.7653668601564044, -7.151571834729253e-09)--(0.5411961047199719, -7.151571834729253e-09)--(0.38268343951666106, -0.15851267235488264)--(0.38268343951666095, -0.3826834277913151)--(0.5411961047199718, -0.541196092994626)--(0.7653668601564043, -0.541196092994626)--(0.9238795253597152, -0.3826834277913152);
		\draw [ultra thick] (0.5411961019689742, -0.5411960918551251)--(0.38268344065616167, -0.3826834305423125); 
		\draw [ultra thick] (0.9238795242202148, -0.15851266960388447)--(0.9238795242202148, -0.3826834305423125); 
		\draw [ultra thick] (0.5411961019689742, 8.291071940114136e-09)--(0.7653668629074022, 8.291071940114136e-09); 
		\draw [dashed] (0.65, 0) -- (0.65, -0.27) -- (0.46, -0.46);
		\draw [dashed] (0.65, -0.27) -- (0.923, -0.27);
		\node [left] at  (0.9238795242202148, 0.3826834305423125) {$C_1$};
		\node [below right] at (0.9238795253597152, -0.3826834277913152) {$C_{0}$};
		\node [below right] at (0.3826834305423126, -0.9238795242202148) {$C_{-1}$};
		\radvertex{ 0.9238795242202148, 0.3826834305423125}{0.02}\radvertex{ 0.9238795253597152, -0.3826834277913152}{0.02}\radvertex{0.3826834305423126, -0.9238795242202148 }{0.02}
		\draw (1.1, -0.38) -- (1.2, -0.38) -- (1.2, 0.38) -- (1.1, 0.38);
		\node [right] at (1.2, 0) {$L_0$};
		\draw (0.38, 1.1) -- (0.38, 1.2) -- (-0.38, 1.2) -- (-0.38, 1.1);
		\node [above] at (0, 1.2) {$L_N$};
		%\draw [ultra thick] (-0.3826834305423126, -0.9238795242202148) -- (-0.9238795253597152, -0.3826834277913152);
		\draw (-.45, -0.98) -- (-0.52, -1.05) -- (-1.05, -0.52) -- (-0.98, -0.45);
		\node [below left] at (-0.77, -0.77) {$L_{3N - 1}$};
		\draw [dotted]  (0.9238795242202148, 0) -- (0,0) -- (0,0.9238795242202148);
		\draw[dotted] (0,0) -- (-0.6532814765755152,-0.6532814765755152);
	\end{tikzpicture}
\caption{The map $\tilde\psi_0$ takes $L_N$ and $L_{3N-1}$ to the sides of the cell $\tilde\psi_0(F_0)$ that intersect other cells of the same size, while $\tilde\psi_0(L_0)\subset L_0$ (thick lines).  The graph $G_0$ (dotted lines) has vertices at the center of the diagram and the centers of the lines $L_0$, $L_N$ and $L_{3N-1}$; $\tilde\psi_0(G_0)$ is shown with dashed lines.}\label{GraphDefHelp}
\end{figure}
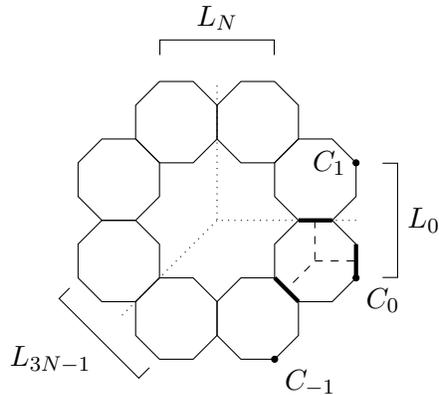

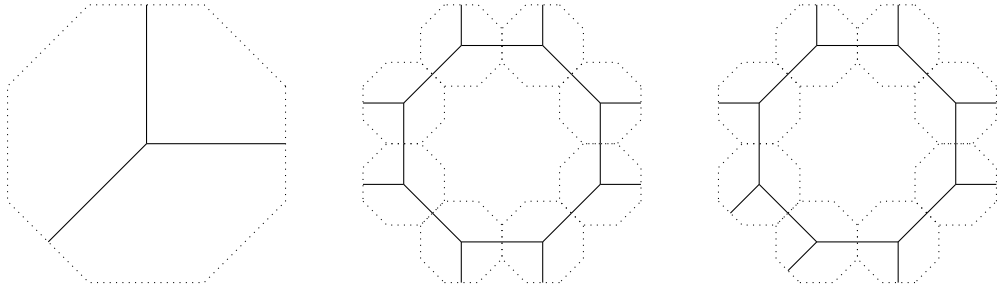
\begin{figure} 
\centering
\begin{tikzpicture}[scale=2]\draw (0.9238795, 0.0)--(0.0, 0.0);
\draw (5.551115e-17, 0.9238795)--(0.0, 0.0);
\draw (-0.6532815, -0.6532815)--(0.0, 0.0);\draw[dotted] (0.9238795, 0.38268343)--(0.38268343, 0.9238795)--(-0.38268343, 0.9238795)--(-0.9238795, 0.38268343)--(-0.9238795, -0.38268343)--(-0.38268343, -0.9238795)--(0.38268343, -0.9238795)--(0.9238795, -0.38268343)--(0.9238795, 0.38268343); \end{tikzpicture}\hspace{0.4in}\begin{tikzpicture}[scale=2]\draw (0.9238795242202148, -0.2705980500730985)--(0.6532814824381883, -0.2705980500730985);
\draw (0.6532814824381883, -8.291071940114136e-09)--(0.6532814824381883, -0.2705980500730985);
\draw (0.4619397582192208, -0.46193977429206595)--(0.6532814824381883, -0.2705980500730985);
\draw (0.9238795242202148, 0.2705980500730985)--(0.6532814824381883, 0.2705980500730985);
\draw (0.6532814824381883, 8.291071940114136e-09)--(0.6532814824381883, 0.2705980500730985);
\draw (0.4619397582192208, 0.46193977429206595)--(0.6532814824381883, 0.2705980500730985);
\draw (0.27059805007309856, 0.9238795242202148)--(0.27059805007309856, 0.6532814824381883);
\draw (8.291071995625288e-09, 0.6532814824381883)--(0.27059805007309856, 0.6532814824381883);
\draw (0.461939774292066, 0.4619397582192208)--(0.27059805007309856, 0.6532814824381883);
\draw (-0.27059805007309845, 0.9238795242202148)--(-0.27059805007309845, 0.6532814824381883);
\draw (-8.291071940114136e-09, 0.6532814824381883)--(-0.27059805007309845, 0.6532814824381883);
\draw (-0.4619397742920659, 0.4619397582192208)--(-0.27059805007309845, 0.6532814824381883);
\draw (-0.9238795242202148, 0.2705980500730986)--(-0.6532814824381883, 0.27059805007309856);
\draw (-0.6532814824381883, 8.291072051136439e-09)--(-0.6532814824381883, 0.27059805007309856);
\draw (-0.46193975821922073, 0.461939774292066)--(-0.6532814824381883, 0.27059805007309856);
\draw (-0.9238795242202148, -0.2705980500730984)--(-0.6532814824381883, -0.27059805007309845);
\draw (-0.6532814824381883, -8.291071884602985e-09)--(-0.6532814824381883, -0.27059805007309845);
\draw (-0.4619397582192209, -0.4619397742920659)--(-0.6532814824381883, -0.27059805007309845);
\draw (-0.27059805007309895, -0.9238795242202146)--(-0.2705980500730989, -0.6532814824381881);
\draw (-8.291072384203346e-09, -0.6532814824381882)--(-0.2705980500730989, -0.6532814824381881);
\draw (-0.46193977429206634, -0.4619397582192206)--(-0.2705980500730989, -0.6532814824381881);
\draw (0.2705980500730986, -0.9238795242202147)--(0.2705980500730987, -0.6532814824381882);
\draw (8.29107210664759e-09, -0.6532814824381881)--(0.2705980500730987, -0.6532814824381882);
\draw (0.4619397742920662, -0.46193975821922073)--(0.2705980500730987, -0.6532814824381882);\draw[dotted] (0.9238795242202148, 0.3826834305423125)--(0.7653668629074022, 0.5411960918551251)--(0.5411961019689742, 0.5411960918551251)--(0.38268344065616167, 0.3826834305423125)--(0.38268344065616167, 0.15851266960388447)--(0.5411961019689742, 8.291071940114136e-09)--(0.7653668629074022, 8.291071940114136e-09)--(0.9238795242202148, 0.15851266960388447)--(0.9238795242202148, 0.3826834305423125);
\draw[dotted] (0.3826834305423126, 0.9238795242202148)--(0.5411960918551251, 0.7653668629074022)--(0.5411960918551251, 0.5411961019689742)--(0.3826834305423125, 0.38268344065616167)--(0.15851266960388452, 0.38268344065616167)--(8.291071995625288e-09, 0.5411961019689742)--(8.291071995625288e-09, 0.7653668629074022)--(0.15851266960388452, 0.9238795242202148)--(0.3826834305423126, 0.9238795242202148);
\draw[dotted] (-0.38268343054231246, 0.9238795242202148)--(-0.541196091855125, 0.7653668629074022)--(-0.541196091855125, 0.5411961019689742)--(-0.38268343054231246, 0.38268344065616167)--(-0.15851266960388447, 0.38268344065616167)--(-8.291071940114136e-09, 0.5411961019689742)--(-8.291071940114136e-09, 0.7653668629074022)--(-0.1585126696038844, 0.9238795242202148)--(-0.38268343054231246, 0.9238795242202148);
\draw[dotted] (-0.9238795242202148, 0.38268343054231263)--(-0.7653668629074022, 0.5411960918551251)--(-0.5411961019689742, 0.5411960918551251)--(-0.38268344065616167, 0.3826834305423126)--(-0.38268344065616167, 0.15851266960388452)--(-0.5411961019689743, 8.291072051136439e-09)--(-0.7653668629074023, 8.291072051136439e-09)--(-0.9238795242202148, 0.15851266960388458)--(-0.9238795242202148, 0.38268343054231263);
\draw[dotted] (-0.9238795242202148, -0.3826834305423124)--(-0.7653668629074024, -0.541196091855125)--(-0.5411961019689744, -0.541196091855125)--(-0.3826834406561618, -0.38268343054231246)--(-0.3826834406561618, -0.1585126696038844)--(-0.5411961019689743, -8.291071884602985e-09)--(-0.7653668629074023, -8.291071884602985e-09)--(-0.9238795242202148, -0.15851266960388435)--(-0.9238795242202148, -0.3826834305423124);
\draw[dotted] (-0.3826834277913151, -0.9238795253597152)--(-0.15851267235488267, -0.9238795253597152)--(-7.151571779218102e-09, -0.7653668601564043)--(-7.151571779218102e-09, -0.5411961047199718)--(-0.15851267235488267, -0.38268343951666095)--(-0.38268342779131514, -0.38268343951666106)--(-0.541196092994626, -0.5411961047199718)--(-0.541196092994626, -0.7653668601564043)--(-0.3826834277913151, -0.9238795253597152);
\draw[dotted] (0.3826834277913149, -0.9238795253597153)--(0.15851267235488245, -0.9238795253597153)--(7.151571557173497e-09, -0.7653668601564044)--(7.151571501662346e-09, -0.541196104719972)--(0.1585126723548824, -0.38268343951666106)--(0.38268342779131487, -0.38268343951666106)--(0.5411960929946258, -0.541196104719972)--(0.5411960929946258, -0.7653668601564043)--(0.3826834277913149, -0.9238795253597153);
\draw[dotted] (0.9238795253597152, -0.3826834277913152)--(0.9238795253597152, -0.15851267235488278)--(0.7653668601564044, -7.151571834729253e-09)--(0.5411961047199719, -7.151571834729253e-09)--(0.38268343951666106, -0.15851267235488264)--(0.38268343951666095, -0.3826834277913151)--(0.5411961047199718, -0.541196092994626)--(0.7653668601564043, -0.541196092994626)--(0.9238795253597152, -0.3826834277913152); \end{tikzpicture}\hspace{0.4in}\begin{tikzpicture}[scale=2]\draw (0.9238795242202148, -0.2705980500730985)--(0.6532814824381883, -0.2705980500730985);
\draw (0.6532814824381883, -8.291071940114136e-09)--(0.6532814824381883, -0.2705980500730985);
\draw (0.4619397582192208, -0.46193977429206595)--(0.6532814824381883, -0.2705980500730985);
\draw (0.9238795242202148, 0.2705980500730985)--(0.6532814824381883, 0.2705980500730985);
\draw (0.6532814824381883, 8.291071940114136e-09)--(0.6532814824381883, 0.2705980500730985);
\draw (0.4619397582192208, 0.46193977429206595)--(0.6532814824381883, 0.2705980500730985);
\draw (0.27059805007309856, 0.9238795242202148)--(0.27059805007309856, 0.6532814824381883);
\draw (8.291071995625288e-09, 0.6532814824381883)--(0.27059805007309856, 0.6532814824381883);
\draw (0.461939774292066, 0.4619397582192208)--(0.27059805007309856, 0.6532814824381883);
\draw (-0.27059805007309845, 0.9238795242202148)--(-0.27059805007309845, 0.6532814824381883);
\draw (-8.291071940114136e-09, 0.6532814824381883)--(-0.27059805007309845, 0.6532814824381883);
\draw (-0.4619397742920659, 0.4619397582192208)--(-0.27059805007309845, 0.6532814824381883);
\draw (-0.9238795242202148, 0.2705980500730986)--(-0.6532814824381883, 0.27059805007309856);
\draw (-0.6532814824381883, 8.291072051136439e-09)--(-0.6532814824381883, 0.27059805007309856);
\draw (-0.46193975821922073, 0.461939774292066)--(-0.6532814824381883, 0.27059805007309856);
\draw (-0.8446231927580601, -0.4619397603929701)--(-0.6532814824381883, -0.27059805007309845);
\draw (-0.4619397721183167, -0.46193976039297013)--(-0.6532814824381883, -0.27059805007309845);
\draw (-0.6532814824381883, 1.1365217877923328e-08)--(-0.6532814824381883, -0.27059805007309845);
\draw (-0.4619397603929706, -0.8446231927580597)--(-0.2705980500730989, -0.6532814824381881);
\draw (-0.4619397603929705, -0.4619397721183164)--(-0.2705980500730989, -0.6532814824381881);
\draw (1.1365217378322967e-08, -0.6532814824381882)--(-0.2705980500730989, -0.6532814824381881);
\draw (0.2705980500730986, -0.9238795242202147)--(0.2705980500730987, -0.6532814824381882);
\draw (8.29107210664759e-09, -0.6532814824381881)--(0.2705980500730987, -0.6532814824381882);
\draw (0.4619397742920662, -0.46193975821922073)--(0.2705980500730987, -0.6532814824381882);\draw[dotted] (0.9238795242202148, 0.3826834305423125)--(0.7653668629074022, 0.5411960918551251)--(0.5411961019689742, 0.5411960918551251)--(0.38268344065616167, 0.3826834305423125)--(0.38268344065616167, 0.15851266960388447)--(0.5411961019689742, 8.291071940114136e-09)--(0.7653668629074022, 8.291071940114136e-09)--(0.9238795242202148, 0.15851266960388447)--(0.9238795242202148, 0.3826834305423125);
\draw[dotted] (0.3826834305423126, 0.9238795242202148)--(0.5411960918551251, 0.7653668629074022)--(0.5411960918551251, 0.5411961019689742)--(0.3826834305423125, 0.38268344065616167)--(0.15851266960388452, 0.38268344065616167)--(8.291071995625288e-09, 0.5411961019689742)--(8.291071995625288e-09, 0.7653668629074022)--(0.15851266960388452, 0.9238795242202148)--(0.3826834305423126, 0.9238795242202148);
\draw[dotted] (-0.38268343054231246, 0.9238795242202148)--(-0.541196091855125, 0.7653668629074022)--(-0.541196091855125, 0.5411961019689742)--(-0.38268343054231246, 0.38268344065616167)--(-0.15851266960388447, 0.38268344065616167)--(-8.291071940114136e-09, 0.5411961019689742)--(-8.291071940114136e-09, 0.7653668629074022)--(-0.1585126696038844, 0.9238795242202148)--(-0.38268343054231246, 0.9238795242202148);
\draw[dotted] (-0.9238795242202148, 0.38268343054231263)--(-0.7653668629074022, 0.5411960918551251)--(-0.5411961019689742, 0.5411960918551251)--(-0.38268344065616167, 0.3826834305423126)--(-0.38268344065616167, 0.15851266960388452)--(-0.5411961019689743, 8.291072051136439e-09)--(-0.7653668629074023, 8.291072051136439e-09)--(-0.9238795242202148, 0.15851266960388458)--(-0.9238795242202148, 0.38268343054231263);
\draw[dotted] (-0.9238795242202148, -0.3826834305423124)--(-0.7653668629074024, -0.541196091855125)--(-0.5411961019689744, -0.541196091855125)--(-0.3826834406561618, -0.38268343054231246)--(-0.3826834406561618, -0.1585126696038844)--(-0.5411961019689743, -8.291071884602985e-09)--(-0.7653668629074023, -8.291071884602985e-09)--(-0.9238795242202148, -0.15851266960388435)--(-0.9238795242202148, -0.3826834305423124);
\draw[dotted] (-0.3826834277913151, -0.9238795253597152)--(-0.15851267235488267, -0.9238795253597152)--(-7.151571779218102e-09, -0.7653668601564043)--(-7.151571779218102e-09, -0.5411961047199718)--(-0.15851267235488267, -0.38268343951666095)--(-0.38268342779131514, -0.38268343951666106)--(-0.541196092994626, -0.5411961047199718)--(-0.541196092994626, -0.7653668601564043)--(-0.3826834277913151, -0.9238795253597152);
\draw[dotted] (0.3826834277913149, -0.9238795253597153)--(0.15851267235488245, -0.9238795253597153)--(7.151571557173497e-09, -0.7653668601564044)--(7.151571501662346e-09, -0.541196104719972)--(0.1585126723548824, -0.38268343951666106)--(0.38268342779131487, -0.38268343951666106)--(0.5411960929946258, -0.541196104719972)--(0.5411960929946258, -0.7653668601564043)--(0.3826834277913149, -0.9238795253597153);
\draw[dotted] (0.9238795253597152, -0.3826834277913152)--(0.9238795253597152, -0.15851267235488278)--(0.7653668601564044, -7.151571834729253e-09)--(0.5411961047199719, -7.151571834729253e-09)--(0.38268343951666106, -0.15851267235488264)--(0.38268343951666095, -0.3826834277913151)--(0.5411961047199718, -0.541196092994626)--(0.7653668601564043, -0.541196092994626)--(0.9238795253597152, -0.3826834277913152); \end{tikzpicture}
\caption{From left to right, for $N = 2$: $G_0$, $\cup_i \psi_i(G_0)$, and $\cup_i\tilde \psi_i(G_0)$.} \label{psifig} 

\end{figure}

For a fixed $m$ the pre-carpet $F_m$ is a union of cells, each of which is a scaled translated copy of the convex set $F_0$. Our immediate goal is to define graphs that reflect the the adjacency structure of the cells and have boundary on $A_m$ and $B_m$. 
We define maps $\psi_w$, and $\Psi_m$ as follows.

\begin{definition} Define maps 
\begin{equation*}
	\psi_j(z) = \begin{cases}
	\phi_j\circ \theta^j( z) & \text{if $j\equiv 0\mod 2$},\\
	\phi_j\circ \theta^{j-1}(\bar z) &  \text{if $j\equiv 1\mod 2$}.
	\end{cases}
\end{equation*}
If $N$ is even, let:
\begin{equation*}
\tilde \psi_j(z) 
	= \begin{cases} 
		\phi_{3N-1}\circ \theta^{3N-1}(z) & j = 3N - 1, \\
		\phi_{3N} \circ \theta^{3N - 1}(\bar z) & j = 3N, \\ 		
		\psi_j(z) & j \in \Lambda \setminus \{3N, 3N-1\}, \\
		\end{cases}
\end{equation*}
and if $N$ is odd, let: 
\begin{equation*}
\tilde \psi_j(z)
	= \begin{cases}  \phi_{N} \circ \theta^{N}(z) & j = N, \\ 
		\phi_{N+1}\circ \theta^{N}(\bar z) & j = N+1, \\
		\psi_j(z) & j \in \Lambda \setminus \{N, N+1\}. \\ 
		\end{cases}
\end{equation*}
Finally, given a word $w=w_1w_2\dotsm w_m$, set $\psi_w =\psi_{w_1}\circ\tilde\psi_{w_2}\circ\dotsm\circ\tilde\psi_{w_m}$ and define a map $\Psi_m=\cup_{|w|=m}\psi_w:F_0\to F_m$.
\end{definition}

\begin{remark} 
Both $\psi_j$ and $\tilde\psi_j$ take $F_0$ to the unique cell of $F_1$ that contains $C_j$, however we first rotate and/or reflect $F_0$ so as to adjust the location of the image of some specific sides $L_k$.   The reason for doing so comes from the adjacency structure of cells in $F_m$ and the location of our desired graph boundary on $A_m\cup B_m$.

Specifically, the choice of rotations and reflections for the $\tilde\psi_j$ ensures that $\tilde\psi_j(L_N)$ and $\tilde\psi_j(L_{3N-1})$ are the sides of the cell $\tilde\psi_j(F_0)$ that intersect neighboring cells, and that if this cell intersects $L_0$, $L_N$ or $L_{3N-1}$ then it does so along the side $\tilde\psi_j(L_0)$, see Figure~\ref{GraphDefHelp}.  In consequence, if we take any connected graph having boundary at either the center or the endpoints of the sides $L_0$, $L_N$ and $L_{3N-1}$ then the union of the images under $\tilde\psi_j$ for $j=0,\dotsc,4N-1$ is also a connected graph with boundary at the same points. This latter is illustrated in Figure~\ref{GraphDefHelp} for the case of the graph $G_0$ of Definition~\ref{defn:G0} using dotted and dashed lines.  The procedure can then be iterated, so the map $\cup_{|w|=m}\tilde\psi_{w_1}\circ\dotsm\circ\tilde\psi_{w_m}$ will also produce a graph of the same type.

The reason for choosing slightly different rotations and reflections for the maps $\psi_j$ is that our final goal is to obtain graphs that reflect the adjacency structure of $F_m$ and have boundary on $A_m\cup B_m$. Using only the $\tilde\psi_j$ would give the correct adjacency structure but  (as discussed above) boundary on the sides $L_0$, $L_N$ and $L_{N-1}$.  To fix this we use  the maps $\psi_j$, each of which maps $L_N$ and $L_{3N-1}$ to the sides where the cell $\psi_j(F_0)$ meets its neighbors, but also has $\psi_j(L_0)\subset A_1\cup B_1$.  Applying a single copy of $\psi_j$ at the end of each of the compositions defining $\Psi_m$ ensures the boundary is in $A_m\cup B_m$.  The effect when starting with the graph $G_0$ from Definition~\ref{defn:G0} is in Figure~\ref{fig:Ggraphs}. Figure~\ref{fig:Dgraphs} shows both the $D_0$ graphs of Definition~\ref{defn:D0} (on the left), which have boundary on $L_0$, $L_N$ and $L_{3N-1}$, and the $D_1$ and $D_2$ graphs (center and right) which have boundary on $A_1\cup B_1$ and $A_2\cup B_2$ respectively.
\end{remark}

We can now define the graphs $G_m$ that we will use to approximate currents on $F_m$.  They correspond to ignoring the internal structure of cells and recording only the (net) flux through the intersections of pairs of cells. Figure~\ref{GraphDefHelp} shows the edges of $G_0$ as dotted lines.  Figure~\ref{fig:Ggraphs} shows graphs $G_2$ on the octacarpet and dodecacarpet.

\begin{definition}\label{defn:G0}
The graph $G_0$ has vertices at $0$ (the center of $F_0$) and at   $\frac12(C_j+C_{j+1})$ for $j=0,N,3N-1$, which are the midpoints of the sides $L_0\cap F_0$, $L_N\cap F_0$ and $L_{3N-1}\cap F_0$.  It has one edge from $0$ to each of the other three vertices. The graphs $G_m$ are defined via $G_m = \Psi_m(G_0)$.
\end{definition}

Let $\tilde I_m^G$ be the optimal current from $A_m$ to $B_m$ on $G_m$, where we recall that these sets were defined in~\eqref{eqn:AnBn}. By symmetry, the total flux through each $L_i \cap A_m$ is $-1/N$, and the total flux through each $L_i \cap B_m$ is $1/N$. Since there are $N$ sides in $A_m$ and also in $B_m$, this gives unit total flux from $A_m$ to $B_m$.  The corresponding optimal potential is denoted $\tilde{u}_m^G$ and is $0$ on $A_m$ and $1$ on $B_m$.   We write $R_m^G$ for the resistance of $G_m$ defined as in~\eqref{eq:R_G}.  Also note that $\tilde{I}_m^G\circ\psi_w$ is a current on $G_0$ for each word $w$ of length $m$.   

\begin{figure}
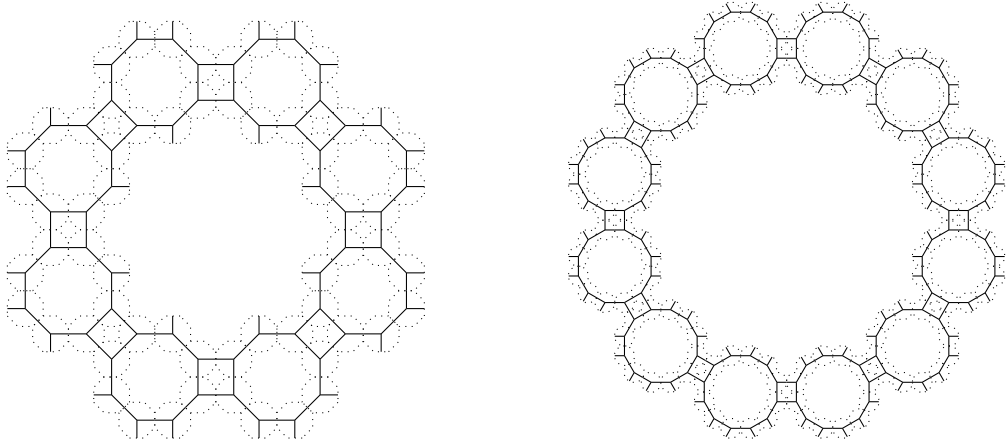

\centering
\begin{tikzpicture}[scale=3]\input{fullocta}\input{octaS2}\end{tikzpicture}\hspace{0.75in}\begin{tikzpicture}[scale=3]\input{fulldodeca}\input{dodecaS2}
\end{tikzpicture}
\caption{The graphs $G_2$ for the octacarpet ($N=2$) and the dodecacarpet ($N = 3$). }
\label{fig:Ggraphs}
\end{figure}

The graphs $D_m$ that we use to approximate potentials on $F_m$ have vertices at each endpoint of a side common to two cells. Figure~\ref{fig:Dgraphs} shows the first few $D_n$ for $N=2$ and $N=3$. 

\begin{definition}\label{defn:D0}
The graph $D_0$ has vertices $\{0,C_0,C_1,C_N,C_{N+1},C_{3N-1},C_{3N}\}$ and edges from $0$ to each of the other six vertices.The graphs $D_m$ are defined by $\Psi_m(D_0)$. Figure~\ref{fig:Dgraphs} shows these graphs for $n = 0, 1, 2$ on the octacarpet and dodecacarpet. 
\end{definition}

\begin{figure}
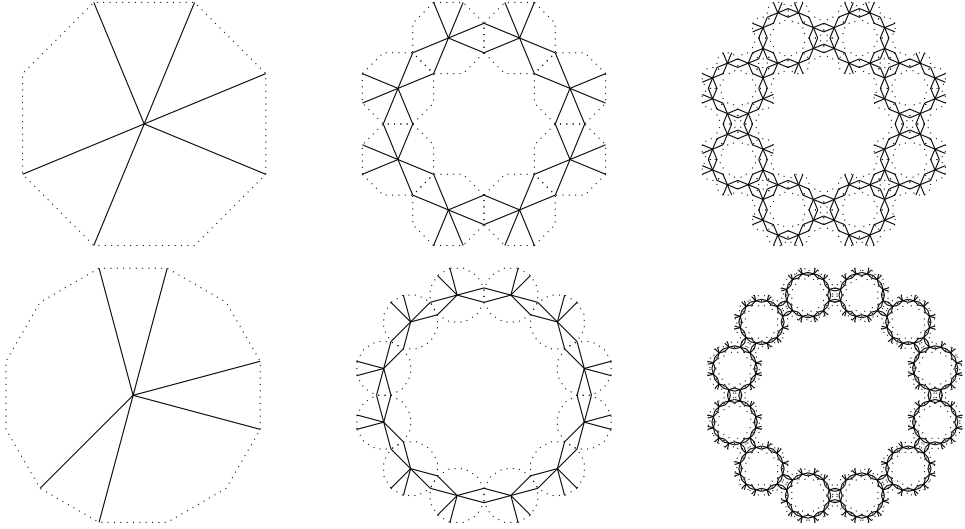

\centering
\begin{tikzpicture}[scale=1.75, rotate = 315]\draw (0.9238795, 0.38268343)--(0.0, 0.0);
\draw (0.38268343, 0.9238795)--(0.0, 0.0);
\draw (-0.38268343, 0.9238795)--(0.0, 0.0);
\draw (-0.9238795, 0.38268343)--(0.0, 0.0);
\draw (-0.38268343, -0.9238795)--(0.0, 0.0);
\draw (0.38268343, -0.9238795)--(0.0, 0.0);\draw[dotted] (0.9238795, 0.38268343)--(0.38268343, 0.9238795)--(-0.38268343, 0.9238795)--(-0.9238795, 0.38268343)--(-0.9238795, -0.38268343)--(-0.38268343, -0.9238795)--(0.38268343, -0.9238795)--(0.9238795, -0.38268343)--(0.9238795, 0.38268343);\end{tikzpicture}\hspace{0.5in}\begin{tikzpicture}[scale=1.75, rotate =45]\draw (0.9238795242202148, 0.3826834305423125)--(0.6532814824381883, 0.2705980500730985);
\draw (0.7653668629074022, 0.5411960918551251)--(0.6532814824381883, 0.2705980500730985);
\draw (0.5411961019689742, 0.5411960918551251)--(0.6532814824381883, 0.2705980500730985);
\draw (0.38268344065616167, 0.3826834305423125)--(0.6532814824381883, 0.2705980500730985);
\draw (0.5411961019689742, 8.291071940114136e-09)--(0.6532814824381883, 0.2705980500730985);
\draw (0.7653668629074022, 8.291071940114136e-09)--(0.6532814824381883, 0.2705980500730985);
\draw (0.3826834305423126, 0.9238795242202148)--(0.27059805007309856, 0.6532814824381883);
\draw (0.5411960918551251, 0.7653668629074022)--(0.27059805007309856, 0.6532814824381883);
\draw (0.5411960918551251, 0.5411961019689742)--(0.27059805007309856, 0.6532814824381883);
\draw (0.3826834305423125, 0.38268344065616167)--(0.27059805007309856, 0.6532814824381883);
\draw (8.291071995625288e-09, 0.5411961019689742)--(0.27059805007309856, 0.6532814824381883);
\draw (8.291071995625288e-09, 0.7653668629074022)--(0.27059805007309856, 0.6532814824381883);
\draw (-0.38268343054231246, 0.9238795242202148)--(-0.27059805007309845, 0.6532814824381883);
\draw (-0.541196091855125, 0.7653668629074022)--(-0.27059805007309845, 0.6532814824381883);
\draw (-0.541196091855125, 0.5411961019689742)--(-0.27059805007309845, 0.6532814824381883);
\draw (-0.38268343054231246, 0.38268344065616167)--(-0.27059805007309845, 0.6532814824381883);
\draw (-8.291071940114136e-09, 0.5411961019689742)--(-0.27059805007309845, 0.6532814824381883);
\draw (-8.291071940114136e-09, 0.7653668629074022)--(-0.27059805007309845, 0.6532814824381883);
\draw (-0.9238795242202148, 0.38268343054231263)--(-0.6532814824381883, 0.27059805007309856);
\draw (-0.7653668629074022, 0.5411960918551251)--(-0.6532814824381883, 0.27059805007309856);
\draw (-0.5411961019689742, 0.5411960918551251)--(-0.6532814824381883, 0.27059805007309856);
\draw (-0.38268344065616167, 0.3826834305423126)--(-0.6532814824381883, 0.27059805007309856);
\draw (-0.5411961019689743, 8.291072051136439e-09)--(-0.6532814824381883, 0.27059805007309856);
\draw (-0.7653668629074023, 8.291072051136439e-09)--(-0.6532814824381883, 0.27059805007309856);
\draw (-0.9238795242202148, -0.3826834305423124)--(-0.6532814824381883, -0.27059805007309845);
\draw (-0.7653668629074024, -0.541196091855125)--(-0.6532814824381883, -0.27059805007309845);
\draw (-0.5411961019689744, -0.541196091855125)--(-0.6532814824381883, -0.27059805007309845);
\draw (-0.3826834406561618, -0.38268343054231246)--(-0.6532814824381883, -0.27059805007309845);
\draw (-0.5411961019689743, -8.291071884602985e-09)--(-0.6532814824381883, -0.27059805007309845);
\draw (-0.7653668629074023, -8.291071884602985e-09)--(-0.6532814824381883, -0.27059805007309845);
\draw (-0.38268343054231296, -0.9238795242202146)--(-0.2705980500730989, -0.6532814824381881);
\draw (-0.5411960918551254, -0.765366862907402)--(-0.2705980500730989, -0.6532814824381881);
\draw (-0.5411960918551254, -0.5411961019689739)--(-0.2705980500730989, -0.6532814824381881);
\draw (-0.38268343054231285, -0.38268344065616156)--(-0.2705980500730989, -0.6532814824381881);
\draw (-8.291072328692195e-09, -0.5411961019689742)--(-0.2705980500730989, -0.6532814824381881);
\draw (-8.291072384203346e-09, -0.7653668629074022)--(-0.2705980500730989, -0.6532814824381881);
\draw (0.38268343054231263, -0.9238795242202147)--(0.2705980500730987, -0.6532814824381882);
\draw (0.5411960918551252, -0.7653668629074022)--(0.2705980500730987, -0.6532814824381882);
\draw (0.5411960918551253, -0.5411961019689742)--(0.2705980500730987, -0.6532814824381882);
\draw (0.3826834305423127, -0.38268344065616167)--(0.2705980500730987, -0.6532814824381882);
\draw (8.291072162158741e-09, -0.5411961019689742)--(0.2705980500730987, -0.6532814824381882);
\draw (8.29107210664759e-09, -0.7653668629074022)--(0.2705980500730987, -0.6532814824381882);
\draw (0.9238795242202146, -0.382683430542313)--(0.6532814824381881, -0.2705980500730989);
\draw (0.765366862907402, -0.5411960918551255)--(0.6532814824381881, -0.2705980500730989);
\draw (0.5411961019689739, -0.5411960918551254)--(0.6532814824381881, -0.2705980500730989);
\draw (0.38268344065616156, -0.38268343054231285)--(0.6532814824381881, -0.2705980500730989);
\draw (0.5411961019689742, -8.291072384203346e-09)--(0.6532814824381881, -0.2705980500730989);
\draw (0.7653668629074022, -8.291072439714497e-09)--(0.6532814824381881, -0.2705980500730989);\draw[dotted] (0.9238795242202148, 0.3826834305423125)--(0.7653668629074022, 0.5411960918551251)--(0.5411961019689742, 0.5411960918551251)--(0.38268344065616167, 0.3826834305423125)--(0.38268344065616167, 0.15851266960388447)--(0.5411961019689742, 8.291071940114136e-09)--(0.7653668629074022, 8.291071940114136e-09)--(0.9238795242202148, 0.15851266960388447)--(0.9238795242202148, 0.3826834305423125);
\draw[dotted] (0.3826834305423126, 0.9238795242202148)--(0.5411960918551251, 0.7653668629074022)--(0.5411960918551251, 0.5411961019689742)--(0.3826834305423125, 0.38268344065616167)--(0.15851266960388452, 0.38268344065616167)--(8.291071995625288e-09, 0.5411961019689742)--(8.291071995625288e-09, 0.7653668629074022)--(0.15851266960388452, 0.9238795242202148)--(0.3826834305423126, 0.9238795242202148);
\draw[dotted] (-0.38268343054231246, 0.9238795242202148)--(-0.541196091855125, 0.7653668629074022)--(-0.541196091855125, 0.5411961019689742)--(-0.38268343054231246, 0.38268344065616167)--(-0.15851266960388447, 0.38268344065616167)--(-8.291071940114136e-09, 0.5411961019689742)--(-8.291071940114136e-09, 0.7653668629074022)--(-0.1585126696038844, 0.9238795242202148)--(-0.38268343054231246, 0.9238795242202148);
\draw[dotted] (-0.9238795242202148, 0.38268343054231263)--(-0.7653668629074022, 0.5411960918551251)--(-0.5411961019689742, 0.5411960918551251)--(-0.38268344065616167, 0.3826834305423126)--(-0.38268344065616167, 0.15851266960388452)--(-0.5411961019689743, 8.291072051136439e-09)--(-0.7653668629074023, 8.291072051136439e-09)--(-0.9238795242202148, 0.15851266960388458)--(-0.9238795242202148, 0.38268343054231263);
\draw[dotted] (-0.9238795242202148, -0.3826834305423124)--(-0.7653668629074024, -0.541196091855125)--(-0.5411961019689744, -0.541196091855125)--(-0.3826834406561618, -0.38268343054231246)--(-0.3826834406561618, -0.1585126696038844)--(-0.5411961019689743, -8.291071884602985e-09)--(-0.7653668629074023, -8.291071884602985e-09)--(-0.9238795242202148, -0.15851266960388435)--(-0.9238795242202148, -0.3826834305423124);
\draw[dotted] (-0.3826834277913151, -0.9238795253597152)--(-0.15851267235488267, -0.9238795253597152)--(-7.151571779218102e-09, -0.7653668601564043)--(-7.151571779218102e-09, -0.5411961047199718)--(-0.15851267235488267, -0.38268343951666095)--(-0.38268342779131514, -0.38268343951666106)--(-0.541196092994626, -0.5411961047199718)--(-0.541196092994626, -0.7653668601564043)--(-0.3826834277913151, -0.9238795253597152);
\draw[dotted] (0.3826834277913149, -0.9238795253597153)--(0.15851267235488245, -0.9238795253597153)--(7.151571557173497e-09, -0.7653668601564044)--(7.151571501662346e-09, -0.541196104719972)--(0.1585126723548824, -0.38268343951666106)--(0.38268342779131487, -0.38268343951666106)--(0.5411960929946258, -0.541196104719972)--(0.5411960929946258, -0.7653668601564043)--(0.3826834277913149, -0.9238795253597153);
\draw[dotted] (0.9238795253597152, -0.3826834277913152)--(0.9238795253597152, -0.15851267235488278)--(0.7653668601564044, -7.151571834729253e-09)--(0.5411961047199719, -7.151571834729253e-09)--(0.38268343951666106, -0.15851267235488264)--(0.38268343951666095, -0.3826834277913151)--(0.5411961047199718, -0.541196092994626)--(0.7653668601564043, -0.541196092994626)--(0.9238795253597152, -0.3826834277913152);\end{tikzpicture}\hspace{0.5in}\begin{tikzpicture}[scale=1.75]\input{fullTfanocta}\input{octaS2}\end{tikzpicture}

\vspace{0.1in}

\begin{tikzpicture}[scale=1.75, rotate = 330]\draw (0.9659258, 0.25881904)--(0.0, 0.0);
\draw (0.70710677, 0.70710677)--(0.0, 0.0);
\draw (-0.25881904, 0.9659258)--(0.0, 0.0);
\draw (-0.70710677, 0.70710677)--(0.0, 0.0);
\draw (-0.25881904, -0.9659258)--(0.0, 0.0);
\draw (0.25881904, -0.9659258)--(0.0, 0.0);\draw[dotted] (0.9659258, 0.25881904)--(0.70710677, 0.70710677)--(0.25881904, 0.9659258)--(-0.25881904, 0.9659258)--(-0.70710677, 0.70710677)--(-0.9659258, 0.25881904)--(-0.9659258, -0.25881904)--(-0.70710677, -0.70710677)--(-0.25881904, -0.9659258)--(0.25881904, -0.9659258)--(0.70710677, -0.70710677)--(0.9659258, -0.25881904)--(0.9659258, 0.25881904);\end{tikzpicture}\hspace{0.5in}\begin{tikzpicture}[scale=1.75, rotate=30]\draw (0.9659258234218514, 0.2588190447926765)--(0.7618016810571369, 0.2041241452319315);
\draw (0.9112309238611065, 0.35355338803590114)--(0.7618016810571369, 0.2041241452319315);
\draw (0.7071067814963918, 0.4082482875966461)--(0.7618016810571369, 0.2041241452319315);
\draw (0.6123724382531672, 0.35355338803590114)--(0.7618016810571369, 0.2041241452319315);
\draw (0.7071067814963918, 2.8672168528309783e-09)--(0.7618016810571369, 0.2041241452319315);
\draw (0.8164965806178818, 2.8672168528309783e-09)--(0.7618016810571369, 0.2041241452319315);
\draw (0.7071067794846062, 0.7071067788583869)--(0.5576775358252053, 0.5576775358252053);
\draw (0.7618016775637009, 0.6123724347597311)--(0.5576775358252053, 0.5576775358252053);
\draw (0.707106778858387, 0.4082482921658044)--(0.5576775358252053, 0.5576775358252053);
\draw (0.6123724347597312, 0.35355339408670966)--(0.5576775358252053, 0.5576775358252053);
\draw (0.35355339323127855, 0.612372434530519)--(0.5576775358252053, 0.5576775358252053);
\draw (0.40824829279202357, 0.7071067794846061)--(0.5576775358252053, 0.5576775358252053);
\draw (0.2588190439372453, 0.9659258236510635)--(0.2041241452319315, 0.7618016810571369);
\draw (0.14942924629740562, 0.9659258227956323)--(0.2041241452319315, 0.7618016810571369);
\draw (2.6380047590812694e-09, 0.8164965797624506)--(0.2041241452319315, 0.7618016810571369);
\draw (3.4934359294247486e-09, 0.7071067821226109)--(0.2041241452319315, 0.7618016810571369);
\draw (0.35355338826511323, 0.612372437397736)--(0.2041241452319315, 0.7618016810571369);
\draw (0.4082482878258582, 0.7071067823518229)--(0.2041241452319315, 0.7618016810571369);
\draw (-0.25881904393724514, 0.9659258236510635)--(-0.2041241452319314, 0.7618016810571369);
\draw (-0.14942924629740545, 0.9659258227956323)--(-0.2041241452319314, 0.7618016810571369);
\draw (-2.638004703570118e-09, 0.8164965797624506)--(-0.2041241452319314, 0.7618016810571369);
\draw (-3.4934358739135973e-09, 0.7071067821226109)--(-0.2041241452319314, 0.7618016810571369);
\draw (-0.3535533882651132, 0.612372437397736)--(-0.2041241452319314, 0.7618016810571369);
\draw (-0.40824828782585815, 0.7071067823518231)--(-0.2041241452319314, 0.7618016810571369);
\draw (-0.707106779484606, 0.7071067788583871)--(-0.5576775358252053, 0.5576775358252053);
\draw (-0.7618016775637008, 0.6123724347597312)--(-0.5576775358252053, 0.5576775358252053);
\draw (-0.707106778858387, 0.40824829216580455)--(-0.5576775358252053, 0.5576775358252053);
\draw (-0.6123724347597312, 0.3535533940867098)--(-0.5576775358252053, 0.5576775358252053);
\draw (-0.3535533932312785, 0.612372434530519)--(-0.5576775358252053, 0.5576775358252053);
\draw (-0.40824829279202346, 0.7071067794846061)--(-0.5576775358252053, 0.5576775358252053);
\draw (-0.9659258234218514, 0.25881904479267676)--(-0.7618016810571367, 0.2041241452319317);
\draw (-0.9112309238611064, 0.35355338803590136)--(-0.7618016810571367, 0.2041241452319317);
\draw (-0.7071067814963917, 0.40824828759664633)--(-0.7618016810571367, 0.2041241452319317);
\draw (-0.612372438253167, 0.3535533880359013)--(-0.7618016810571367, 0.2041241452319317);
\draw (-0.7071067814963917, 2.867217074875583e-09)--(-0.7618016810571367, 0.2041241452319317);
\draw (-0.8164965806178818, 2.867217074875583e-09)--(-0.7618016810571367, 0.2041241452319317);
\draw (-0.9659258234218514, -0.25881904479267653)--(-0.7618016810571369, -0.20412414523193154);
\draw (-0.9112309238611065, -0.35355338803590114)--(-0.7618016810571369, -0.20412414523193154);
\draw (-0.7071067814963918, -0.4082482875966462)--(-0.7618016810571369, -0.20412414523193154);
\draw (-0.6123724382531672, -0.3535533880359012)--(-0.7618016810571369, -0.20412414523193154);
\draw (-0.7071067814963918, -2.8672169083421295e-09)--(-0.7618016810571369, -0.20412414523193154);
\draw (-0.8164965806178818, -2.8672169083421295e-09)--(-0.7618016810571369, -0.20412414523193154);
\draw (-0.7071067794846064, -0.7071067788583866)--(-0.5576775358252055, -0.557677535825205);
\draw (-0.7618016775637012, -0.6123724347597308)--(-0.5576775358252055, -0.557677535825205);
\draw (-0.7071067788583872, -0.4082482921658041)--(-0.5576775358252055, -0.557677535825205);
\draw (-0.6123724347597314, -0.3535533940867094)--(-0.5576775358252055, -0.557677535825205);
\draw (-0.3535533932312789, -0.6123724345305189)--(-0.5576775358252055, -0.557677535825205);
\draw (-0.4082482927920239, -0.7071067794846059)--(-0.5576775358252055, -0.557677535825205);
\draw (-0.2588190439372453, -0.9659258236510635)--(-0.2041241452319314, -0.7618016810571369);
\draw (-0.14942924629740562, -0.9659258227956324)--(-0.2041241452319314, -0.7618016810571369);
\draw (-2.638004703570118e-09, -0.8164965797624507)--(-0.2041241452319314, -0.7618016810571369);
\draw (-3.493435818402446e-09, -0.707106782122611)--(-0.2041241452319314, -0.7618016810571369);
\draw (-0.35355338826511307, -0.612372437397736)--(-0.2041241452319314, -0.7618016810571369);
\draw (-0.4082482878258581, -0.7071067823518229)--(-0.2041241452319314, -0.7618016810571369);
\draw (0.258819043937245, -0.9659258236510636)--(0.20412414523193115, -0.7618016810571369);
\draw (0.14942924629740528, -0.9659258227956324)--(0.20412414523193115, -0.7618016810571369);
\draw (2.638004426014362e-09, -0.8164965797624507)--(0.20412414523193115, -0.7618016810571369);
\draw (3.493435596357841e-09, -0.707106782122611)--(0.20412414523193115, -0.7618016810571369);
\draw (0.35355338826511284, -0.6123724373977361)--(0.20412414523193115, -0.7618016810571369);
\draw (0.40824828782585787, -0.7071067823518231)--(0.20412414523193115, -0.7618016810571369);
\draw (0.707106779484606, -0.7071067788583871)--(0.5576775358252051, -0.5576775358252054);
\draw (0.7618016775637007, -0.6123724347597312)--(0.5576775358252051, -0.5576775358252054);
\draw (0.7071067788583869, -0.40824829216580455)--(0.5576775358252051, -0.5576775358252054);
\draw (0.612372434759731, -0.3535533940867098)--(0.5576775358252051, -0.5576775358252054);
\draw (0.35355339323127843, -0.6123724345305193)--(0.5576775358252051, -0.5576775358252054);
\draw (0.40824829279202346, -0.7071067794846062)--(0.5576775358252051, -0.5576775358252054);
\draw (0.9659258234218513, -0.2588190447926772)--(0.7618016810571366, -0.20412414523193215);
\draw (0.9112309238611063, -0.35355338803590186)--(0.7618016810571366, -0.20412414523193215);
\draw (0.7071067814963916, -0.4082482875966468)--(0.7618016810571366, -0.20412414523193215);
\draw (0.612372438253167, -0.35355338803590175)--(0.7618016810571366, -0.20412414523193215);
\draw (0.7071067814963917, -2.867217518964793e-09)--(0.7618016810571366, -0.20412414523193215);
\draw (0.8164965806178817, -2.867217518964793e-09)--(0.7618016810571366, -0.20412414523193215);\draw[dotted] (0.9659258234218514, 0.2588190447926765)--(0.9112309238611065, 0.35355338803590114)--(0.8164965806178818, 0.4082482875966461)--(0.7071067814963918, 0.4082482875966461)--(0.6123724382531672, 0.35355338803590114)--(0.5576775386924222, 0.2588190447926765)--(0.5576775386924222, 0.14942924567118648)--(0.6123724382531672, 0.05469490242796185)--(0.7071067814963918, 2.8672168528309783e-09)--(0.8164965806178818, 2.8672168528309783e-09)--(0.9112309238611065, 0.05469490242796185)--(0.9659258234218514, 0.14942924567118648)--(0.9659258234218514, 0.2588190447926765);
\draw[dotted] (0.7071067794846062, 0.7071067788583869)--(0.7618016775637009, 0.612372434759731)--(0.761801678419132, 0.5029826371198913)--(0.707106778858387, 0.4082482921658043)--(0.6123724347597311, 0.3535533940867096)--(0.5029826371198914, 0.3535533932312785)--(0.40824829216580444, 0.40824829279202357)--(0.3535533940867097, 0.5029826368906793)--(0.35355339323127855, 0.6123724345305191)--(0.4082482927920237, 0.7071067794846061)--(0.5029826368906795, 0.7618016775637008)--(0.6123724345305193, 0.7618016784191319)--(0.7071067794846062, 0.7071067788583869);
\draw[dotted] (0.2588190439372453, 0.9659258236510635)--(0.14942924629740562, 0.9659258227956323)--(0.054694902198749756, 0.9112309247165377)--(2.6380047590812694e-09, 0.8164965797624506)--(3.4934359294247486e-09, 0.7071067821226109)--(0.054694901572530624, 0.612372438023955)--(0.14942924652661765, 0.55767753846321)--(0.2588190441664574, 0.5576775393186413)--(0.35355338826511323, 0.612372437397736)--(0.4082482878258582, 0.7071067823518229)--(0.40824828697042703, 0.8164965799916627)--(0.35355338889133237, 0.9112309240903186)--(0.2588190439372453, 0.9659258236510635);
\draw[dotted] (-0.25881904479267637, 0.9659258234218514)--(-0.35355338803590103, 0.9112309238611065)--(-0.40824828759664605, 0.8164965806178819)--(-0.40824828759664605, 0.7071067814963918)--(-0.35355338803590103, 0.6123724382531672)--(-0.2588190447926764, 0.5576775386924222)--(-0.14942924567118643, 0.5576775386924222)--(-0.05469490242796177, 0.6123724382531672)--(-2.867216741808676e-09, 0.7071067814963918)--(-2.867216741808676e-09, 0.8164965806178818)--(-0.05469490242796177, 0.9112309238611065)--(-0.14942924567118637, 0.9659258234218514)--(-0.25881904479267637, 0.9659258234218514);
\draw[dotted] (-0.7071067788583869, 0.7071067794846062)--(-0.612372434759731, 0.7618016775637009)--(-0.5029826371198913, 0.761801678419132)--(-0.4082482921658043, 0.707106778858387)--(-0.3535533940867096, 0.6123724347597311)--(-0.3535533932312785, 0.5029826371198914)--(-0.40824829279202357, 0.40824829216580444)--(-0.5029826368906793, 0.3535533940867097)--(-0.6123724345305191, 0.35355339323127855)--(-0.7071067794846061, 0.4082482927920237)--(-0.7618016775637008, 0.5029826368906795)--(-0.7618016784191319, 0.6123724345305193)--(-0.7071067788583869, 0.7071067794846062);
\draw[dotted] (-0.9659258236510635, 0.25881904393724553)--(-0.9659258227956323, 0.14942924629740578)--(-0.9112309247165375, 0.05469490219874995)--(-0.8164965797624505, 2.638004925614723e-09)--(-0.7071067821226108, 3.4934361514693535e-09)--(-0.612372438023955, 0.0546949015725309)--(-0.55767753846321, 0.14942924652661788)--(-0.5576775393186412, 0.2588190441664576)--(-0.612372437397736, 0.35355338826511346)--(-0.7071067823518229, 0.4082482878258584)--(-0.8164965799916626, 0.40824828697042725)--(-0.9112309240903185, 0.35355338889133253)--(-0.9659258236510635, 0.25881904393724553);
\draw[dotted] (-0.9659258234218514, -0.25881904479267653)--(-0.9112309238611065, -0.35355338803590114)--(-0.8164965806178819, -0.4082482875966462)--(-0.7071067814963918, -0.4082482875966462)--(-0.6123724382531672, -0.3535533880359012)--(-0.5576775386924222, -0.25881904479267653)--(-0.5576775386924222, -0.14942924567118654)--(-0.6123724382531672, -0.054694902427961906)--(-0.7071067814963918, -2.8672169083421295e-09)--(-0.8164965806178818, -2.8672169083421295e-09)--(-0.9112309238611064, -0.05469490242796185)--(-0.9659258234218514, -0.14942924567118648)--(-0.9659258234218514, -0.25881904479267653);
\draw[dotted] (-0.7071067788583874, -0.7071067794846058)--(-0.6123724347597316, -0.7618016775637005)--(-0.5029826371198919, -0.7618016784191317)--(-0.4082482921658048, -0.7071067788583868)--(-0.35355339408671005, -0.6123724347597309)--(-0.3535533932312788, -0.5029826371198912)--(-0.4082482927920238, -0.40824829216580416)--(-0.5029826368906796, -0.35355339408670944)--(-0.6123724345305194, -0.3535533932312782)--(-0.7071067794846063, -0.4082482927920232)--(-0.7618016775637011, -0.502982636890679)--(-0.7618016784191324, -0.6123724345305187)--(-0.7071067788583874, -0.7071067794846058);
\draw[dotted] (-0.25881904393724514, -0.9659258236510635)--(-0.14942924629740545, -0.9659258227956323)--(-0.05469490219874962, -0.9112309247165375)--(-2.638004703570118e-09, -0.8164965797624506)--(-3.4934358739135973e-09, -0.7071067821226109)--(-0.054694901572530624, -0.612372438023955)--(-0.14942924652661765, -0.55767753846321)--(-0.25881904416645735, -0.5576775393186413)--(-0.3535533882651132, -0.612372437397736)--(-0.40824828782585815, -0.7071067823518231)--(-0.4082482869704269, -0.8164965799916628)--(-0.3535533888913322, -0.9112309240903186)--(-0.25881904393724514, -0.9659258236510635);
\draw[dotted] (0.2588190439372449, -0.9659258236510636)--(0.14942924629740523, -0.9659258227956324)--(0.054694902198749396, -0.9112309247165377)--(2.638004370503211e-09, -0.8164965797624507)--(3.493435596357841e-09, -0.7071067821226109)--(0.05469490157253032, -0.6123724380239552)--(0.14942924652661738, -0.5576775384632102)--(0.25881904416645707, -0.5576775393186413)--(0.3535533882651129, -0.6123724373977362)--(0.40824828782585787, -0.7071067823518231)--(0.4082482869704267, -0.8164965799916628)--(0.3535533888913319, -0.9112309240903187)--(0.2588190439372449, -0.9659258236510636);
\draw[dotted] (0.707106779484606, -0.7071067788583871)--(0.7618016775637007, -0.6123724347597312)--(0.7618016784191318, -0.5029826371198916)--(0.7071067788583869, -0.40824829216580455)--(0.612372434759731, -0.3535533940867098)--(0.5029826371198913, -0.35355339323127866)--(0.4082482921658043, -0.4082482927920237)--(0.3535533940867096, -0.5029826368906796)--(0.35355339323127843, -0.6123724345305193)--(0.40824829279202346, -0.7071067794846062)--(0.5029826368906792, -0.761801677563701)--(0.612372434530519, -0.7618016784191322)--(0.707106779484606, -0.7071067788583871);
\draw[dotted] (0.9659258236510634, -0.2588190439372461)--(0.9659258227956322, -0.14942924629740634)--(0.9112309247165375, -0.054694902198750506)--(0.8164965797624506, -2.638005425215084e-09)--(0.7071067821226109, -3.4934365955585633e-09)--(0.612372438023955, -0.054694901572531235)--(0.55767753846321, -0.1494292465266182)--(0.5576775393186411, -0.25881904416645796)--(0.6123724373977357, -0.35355338826511384)--(0.7071067823518227, -0.40824828782585887)--(0.8164965799916625, -0.40824828697042775)--(0.9112309240903183, -0.35355338889133303)--(0.9659258236510634, -0.2588190439372461);\end{tikzpicture}\hspace{0.5in}\begin{tikzpicture}[scale=1.75]\input{fullTfandodeca}\input{dodecaS2}\end{tikzpicture}
\caption{The graphs $D_0$, $D_1$ and $D_2$ for the octacarpet ($N=2$) and the dodecacarpet ($N = 3$). Note that $D_0$ has boundary on the sides $L_0$, $L_N$, and $L_{3N-1}$, whereas $D_1$ and $D_2$ have boundary on $A_1 \cup B_1$ and $A_2 \cup B_2$ respectively. }
\label{fig:Dgraphs}
\end{figure}

We let $\tilde{u}_m^D$ denote the optimal potential on $D_m$ for the boundary conditions $\tilde{u}_m=0$ at vertices in $A_m$ and $\tilde{u}_m=1$ at vertices in $B_m$.  The resistance of $D_m$ is written $R_m^D$.  As with currents, $\tilde{u}_m\circ \psi_w$ is a potential on $D_0$ for any word $w$ of length $m$.

\begin{lemma}\label{lem:RGRD}
For all $m\geq1$, $R_m^G=2R_m^D$.
\end{lemma}
\begin{proof}
Each edge in $G_m$ connects the center $x$ of a cell to a point $y$ on a side of the cell.  Writing $y_\pm$ for the endpoints of that side we see that there are two edges in $D_m$ connecting $x$ to the same side at $y_\pm$.  In this sense, each edge of $G_m$ corresponds to two edges of $D_m$ and conversely.

From the optimal potential $\tilde{u}_m^G$ for $G_m$ define a function $f$ on $G_m$ by setting $f(x)=\tilde{u}_m^G(x)$ at cell centers and $f(y_\pm)=\tilde{u}_m^G(y)$ at endpoints $y_\pm$ of a side with center $y$. This ensures $f(y_\pm)-f(x)=f(y)-f(x)$, so that two edges in $D_m$ have the same edge difference as the corresponding single edge in $D_m$. Clearly $f$ is a feasible potential on $D_m$, so $(R_m^D)^{-1}\leq\DF_{D_m}(f,f)=2\DF_{G_m}(\tilde{u}_m^G,\tilde{u}_m^G)=2(R_m^G)^{-1}$.

Conversely, beginning with the optimal potential $\tilde{u}_m^D$ define $f$ on $G_m$ by $f(x)=\tilde{u}_m^D(x)$ at cell centers and $f(y)=\frac12\bigl(\tilde{u}_m^D(y_+)+\tilde{u}_m^D(y_-)\bigr)$ if $y$ is the center of a cell side with endpoints $y_\pm$. The edge difference $f(x)-f(y)$ in $G_m$ is half the sum of the edge difference on the corresponding edges in $D_m$, so using that $f$ is a feasible potential on $G_m$ we have $(R_m^G)^{-1}\leq \DF_{G_m}(f,f)=\frac12\DF_{D_m}(\tilde{u}_m^G,\tilde{u}_m^G)=\frac12(R_m^D)^{-1}$.
\end{proof}

\subsection{Currents and potentials with energy estimates via symmetry}
Fix $n\geq0$ and recall that $\tilde{u}_{F_n}$ denotes the optimal potential on $F_n$ with boundary values $0$ on $A_n$ and $1$ on $B_n$.
 In order to exploit the symmetries of $F_n$ it is convenient to work instead with $u_n=2\tilde{u}_{F_n}-1$; evidently $\DF_{F_n}(u_n,u_n)=4\DF_{F_n}(\tilde{u}_{F_n},\tilde{u}_{F_n})=4R_n^{-1}$ is then minimal for potentials that are $-1$ on $A_n$ and $1$ on $B_n$.  The corresponding current $J_n=R_n\nabla u_n$ minimizes the energy for currents with flux $2$ from $A_n$ to $B_n$ and has $E_{F_n}(J_n,J_n)=4R_n$.  We begin our analysis by recording some symmetry properties of $J_n$.

\begin{lemma}\label{lem:Jtheta}
Both $u_n\circ\theta^2=-u_n$ and $J_n\circ \theta^2=-J_n$.
\end{lemma}
\begin{proof}
The rotation $\theta$ takes $C_j$ to $C_{j+1}$, thus $L_j$ to $L_{j+1}$.  It then follows from the definition~\eqref{eqn:AnBn} of $A_n$ and $B_n$ that $\theta^2$ exchanges $A_n$ and $B_n$; see Figure~\ref{fig:AmBm} for an example in the case $N=3$. It follows that $- u_n \circ \theta^2$ is a feasible potential for the problem optimized by $u_n$ and by symmetry $\cal E_{F_n}(-u_n \circ \theta^2, -u_n \circ \theta^2) = \cal E_{F_n}(u_n, u_n)$,  so $-u_n \circ \theta^2 = u_n$ by uniqueness of  the energy minimizer. The argument for currents is similar.
\end{proof}

One consequence of this lemma is that the flux of $J_n$ through each of the sides $L_j\cap F_n$ in $A_n$ is independent of $j$ and hence equal to $-\frac2N$. Similarly, the flux through each side in $B_n$ is $\frac2N$.

\begin{lemma}\label{lem:Jconj}
Both $u_n(\bar{z})=u_n(z)$ and $J_n(\bar{z})=J_n(z)$.
\end{lemma}
\begin{proof}
Under complex conjugation the point $C_j=\exp\frac{(2j-1)i\pi}{4N}$ is mapped to 
\begin{equation*}
	\bar{C}_j=\exp\frac{(1-2j)i\pi}{4N} = \exp \frac{ (8N-2j+1)i\pi}{4N} = \exp\frac{ (2(4N-j+1)-1)i\pi}{4N} = C_{4N-j+1}
	\end{equation*}
Then the endpoints $C_{4k}$ and $C_{4k+1}$ of $L_{4k}$ are mapped to $C_{4(N-k)+1}$ and $C_{4(N-k)}$ so $L_{4k}$ is mapped to $L_{4(N-k)}$. This shows $A_n$ is invariant under complex conjugation.  Similarly, $\bar{C}_{4k+2}=C_{4(N-k-1)+3}$ and $\bar{C}_{4k+3}=C_{4(N-k-1)+2}$, so $L_{4k+2}$ is mapped to $L_{4(N-k-1)}$ and $B_n$ is invariant under complex conjugation.  Both $u_n$ and  $J_n$ are determined by their boundary data on these sets.
\end{proof}

We decompose $F_n$ into sectors within triangles by taking, for integers $j$ and $j+1$ modulo $4N$, $T_j^*$ to be the interior of the triangle with vertices $\{0,C_j,C_{j+1}\}$, and defining our sectors by $T_j(n) = F_n \cap T_j^*$.  For notational simplicity we will drop the dependence on $n$ and just write $T_j$.  This is shown for $N=2$ in the left image in Figure~\ref{fig:currentconstruction}. Then the central diagram in Figure~\ref{fig:currentconstruction} illustrates the fact that, up to a change of sign, both $u_n$ and $J_n$ have one behavior on sectors $T_j$ with $j$ even, and another behavior on sectors with $j$ odd.  This motivates us to define
\begin{align}
	v^j=\bigl.(u_n\circ \theta^{-j})\bigr|_{T_j} &&    w^j = \bigl.(u_n\circ \theta^{-j+1})\bigr|_{T_j} \label{eqn:vjandwj}\\
	V^j = \bigl.(J_n\circ \theta^{-j})\bigr|_{T_j} &&    W^j = \bigl.(J_n\circ \theta^{-j+1})\bigr|_{T_j}. \label{eqn:VjandWj}
\end{align}
Examples of $V^j$ and $W^j$ in various sectors are shown on the right in Figure~\ref{fig:currentconstruction} for $N=2$.

Symmetry under rotations shows us that the following quantities are independent of $j$
\begin{align}
\DF_n(v) = \int_{T_j}|v^j|^2 && \DF_n(w) = \int_{T_j}|w^j|^2 \label{eqn:EvEw}\\
E_n(V) = \int_{T_j}|V^j|^2 && E_n(W) = \int_{T_j}|W^j|^2 \label{eqn:EVEW}
\end{align}
and therefore that 
\begin{align}
	4R_n^{-1}&= \DF_{F_n}(u_n,u_n) = 2N(\DF_n(v)+\DF_n(w))\label{eqn:Rnbyvw}\\
	4R_n &=E_{F_n}(J_n,J_n) = 2N\bigl( E_n(V)+E_n(W)\bigr).\label{eqn:RnbyVW}
	\end{align}

\begin{figure}
\centering
\begin{tikzpicture}[scale=2, rotate=-22.5]\input{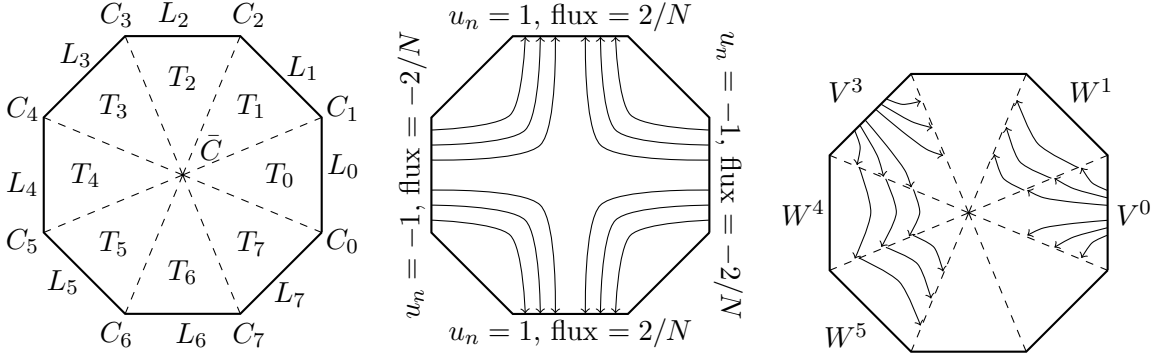}\end{tikzpicture}\hspace{0.1in}\begin{tikzpicture}[scale=2]\draw [thick, rotate = -22.5] (1, 0) -- (0.707, 0.707) -- (0, 1) -- (-0.707, 0.707) -- (-1, 0) -- (-0.707, -0.707) -- (0, -1) -- (0.707, -0.707) -- (1, 0);

\draw [->] plot [smooth] coordinates{ (0.92,0.1) (0.2, 0.2) (0.1, 0.92) }; 
\draw [->] plot [smooth] coordinates{ (0.92,0.2) (0.3, 0.3) (0.2, 0.92) }; 
\draw [->] plot [smooth] coordinates{ (0.92,0.3) (0.4, 0.4) (0.3, 0.92) }; 

\draw [->] plot [smooth] coordinates{ (-0.92,0.1) (-0.2, 0.2) (-0.1, 0.92) }; 
\draw [->] plot [smooth] coordinates{ (-0.92,0.2) (-0.3, 0.3) (-0.2, 0.92) }; 
\draw [->] plot [smooth] coordinates{ (-0.92,0.3) (-0.4, 0.4) (-0.3, 0.92) }; 

\draw [->] plot [smooth] coordinates{ (0.92,-0.1) (0.2, -0.2) (0.1, -0.92) }; 
\draw [->] plot [smooth] coordinates{ (0.92,-0.2) (0.3, -0.3) (0.2, -0.92) }; 
\draw [->] plot [smooth] coordinates{ (0.92,-0.3) (0.4, -0.4) (0.3, -0.92) }; 

\draw [->] plot [smooth] coordinates{ (-0.92,-0.1) (-0.2, -0.2) (-0.1, -0.92) }; 
\draw [->] plot [smooth] coordinates{ (-0.92,-0.2) (-0.3, -0.3) (-0.2, -0.92) }; 
\draw [->] plot [smooth] coordinates{ (-0.92,-0.3) (-0.4, -0.4) (-0.3, -0.92) }; 

\node at (0, 1.05) {$u_n = 1$, flux $= 2/N$};
\node at (0, -1.05) {$u_n = 1$, flux $= 2/N$};

\node [rotate = -90] at (1.05, 0) {$u_n = -1$, flux $= -2/N$};

\node [rotate = 90] at (-1.05, 0) {$u_n = -1$, flux $= -2/N$};\end{tikzpicture}\hspace{0.1in}\begin{tikzpicture}[scale=2]\draw [thick, rotate = -22.5] (1, 0) -- (0.707, 0.707) -- (0, 1) -- (-0.707, 0.707) -- (-1, 0) -- (-0.707, -0.707) -- (0, -1) -- (0.707, -0.707) -- (1, 0);

\draw [dashed, rotate = -22.5] (0, 0) -- (1, 0);
\draw  [dashed, rotate = -22.5] (0, 0) -- (0.707, 0.707);
\draw [dashed, rotate = -22.5]  (0, 0) -- (-1, 0);
\draw [dashed, rotate = -22.5]  (0, 0) -- (-0.707, 0.707);
\draw [dashed, rotate = -22.5] (0, 0) -- (0, -1);
\draw  [dashed, rotate = -22.5] (0, 0) -- (0.707, -0.707);
\draw  [dashed, rotate = -22.5] (0, 0) -- (-0.707, -0.707);
\draw [dashed, rotate = -22.5]  (0, 0) -- (0, 1);

\draw [->] plot [smooth] coordinates{ (0.92, 0.05) (0.7, 0.07) (0.5, 0.1) (0.39, 0.15) }; 
\draw [->] plot [smooth] coordinates{ (0.92, 0.1) (0.7, 0.16) (0.57, 0.23) }; 
\draw [->] plot [smooth] coordinates{ (0.92, 0.15) (0.82, 0.2) (0.75, 0.3) }; 
\draw [->] plot [smooth] coordinates{ (0.92, -0.05) (0.7, -0.07) (0.5, -0.1) (0.39, -0.15) }; 
\draw [->] plot [smooth] coordinates{ (0.92, -0.1) (0.7, -0.16) (0.57, -0.23) }; 
\draw [->] plot [smooth] coordinates{ (0.92, -0.15) (0.82, -0.2) (0.75, -0.3) }; 

\draw [->, rotate = 135] plot [smooth] coordinates{ (0.92, 0.05) (0.7, 0.07) (0.5, 0.1) (0.39, 0.15) }; 
\draw [->, rotate = 135] plot [smooth] coordinates{ (0.92, 0.1) (0.7, 0.16) (0.57, 0.23) }; 
\draw [->, rotate = 135] plot [smooth] coordinates{ (0.92, 0.15) (0.82, 0.2) (0.75, 0.3) }; 
\draw [->, rotate = 135] plot [smooth] coordinates{ (0.92, -0.05) (0.7, -0.07) (0.5, -0.1) (0.39, -0.15) }; 
\draw [->, rotate = 135] plot [smooth] coordinates{ (0.92, -0.1) (0.7, -0.16) (0.57, -0.23) }; 
\draw [->, rotate = 135] plot [smooth] coordinates{ (0.92, -0.15) (0.82, -0.2) (0.75, -0.3) }; 

\draw [->] plot [smooth] coordinates {  (0.37, 0.15) (0.23, 0.23) (0.17, 0.39)};
\draw [->] plot [smooth] coordinates {  (0.55, 0.23) (0.34, 0.34) (0.25, 0.57)};
\draw [->] plot [smooth] coordinates {  (0.73, 0.3) (0.46, 0.46) (0.32, 0.75)};

\draw [->, rotate = 180] plot [smooth] coordinates {  (0.37, 0.15) (0.23, 0.23) (0.17, 0.39)};
\draw [->, rotate = 180] plot [smooth] coordinates {  (0.55, 0.23) (0.34, 0.34) (0.25, 0.57)};
\draw [->, rotate = 180] plot [smooth] coordinates {  (0.73, 0.3) (0.46, 0.46) (0.32, 0.75)};

\draw [->, rotate = 135] plot [smooth] coordinates {  (0.37, 0.15) (0.23, 0.23) (0.17, 0.39)};
\draw [->, rotate = 135] plot [smooth] coordinates {  (0.55, 0.23) (0.34, 0.34) (0.25, 0.57)};
\draw [->, rotate = 135] plot [smooth] coordinates {  (0.73, 0.3) (0.46, 0.46) (0.32, 0.75)};

\node at (1.1, 0) {$V^0$};
\node at (-1.1, 0) {$W^4$};
\node at (0.815, 0.815) {$W^1$};
\node at (-0.815, 0.815) {$V^3$};
\node at (-0.815, -0.815) {$W^5$};\end{tikzpicture}
\caption{For $N=2$: Decomposition of $F_0$ into sectors $T_j$ (left), General current flow lines for $J_n$  (middle),  and examples of $V^j$ and $W^j$ vector fields (right)} \label{fig:currentconstruction}

\end{figure}

\begin{lemma}\label{lem:VorthogW}
For any $j\in\Lambda(N)$, $\int_{T_j} \nabla v^j\cdot \nabla w^j =\int_{T_j} V^j\cdot W^j =0$.
\end{lemma}
\begin{proof}
By rotational symmetry it is enough to verify this for $j=0$. Since $J_n=R_nu_n$, we have $\nabla v^j=R_nV^j$ and $\nabla w^j=R_nW^j$, so we work only with $V^0$ and $W^0$.
The sector $T_0$ is symmetrical under complex conjugation, and using  Lemmas~\ref{lem:Jtheta} and~\ref{lem:Jconj} we have
\begin{gather}
	V^0(\bar{z})= J_n(\bar{z}) = J_n(z)=V^0(z), \label{eqn:V0sym}\\
	W^0(\bar{z})
		= J_n\circ \theta(\bar{z}) 
		=-J_n\circ\theta^{-1}(\bar{z})
		= -J_n(\overline{\theta(z)}) 
		= -J_n(\theta(z)) = -W^0(z). \label{eqn:W0antisym}
	\end{gather}
Thus $V^0\cdot W^0(\bar{z}) = -V^0\cdot W^0(z)$ and the result follows.
\end{proof}

In addition to being orthogonal, the vector fields $V^j$ and $W^j$ have the property that they can easily be glued together to form currents on $F_n$. Recall that to be a current a vector field $J$ must be $L^2$ on the domain and satisfy $\nabla\cdot J=0$.

\begin{lemma}\label{lem:gluingVW}
If $J$ is a vector field such that $J|_{T_l}=\alpha_l V^l+\beta_l W^l$ and $\alpha_{l+1}+\alpha_l=\beta_{l+1}-\beta_l$ for each $l$, then $J$ is a current.
\end{lemma}

\begin{proof}
The fields $V^l$ and $W^l$ are the restriction of currents to the sets $T_l$, thus $J$ is an $L^2$ function on $F_n$ for each $l$. To see the given linear combination is a current we must verify $\nabla\cdot J=0$ using  symmetry considerations that imply the cancellation of the weak divergences as in Lemma~\ref{lem:potandcurglue}(ii).

The symmetry of $V^0$ in~\eqref{eqn:V0sym} shows that $\nabla\cdot V^0$ and $-\nabla\cdot V^1$ cancel on $T_0\cap T_1$, so $V^0-V^1$ is a current on $T_0\cup T_1$.  Similarly, the antisymmetry of $W^0$ in~\eqref{eqn:W0antisym} shows the fluxes of $W^0$ and $W^1$ cancel on $T_0\cap T_1$, so $W^0+W^1$ is a current on $T_0\cup T_1$.  In addition we note that $V^0+W^1=J_n|_{T_0\cup T_1}$ is the restriction of the optimal current and is hence a current on $T_0\cup T_1$.  Combining these with the definition~\eqref{eqn:VjandWj} we see that each of $V^l-V^{l+1}$, $W^l+W^{l+1}$ and $V^l+W^{l+1}$ are currents on $T_l\cup T_{l+1}$ as they are obtained from the $l=0$ case by rotations.

Similarly, the antisymmetry of $W^0$ in~\eqref{eqn:W0antisym} shows the fluxes of $W^0$ and $W^1$ cancel on $T_0\cap T_1$, so $W^0+W^1$ is a current on $T_0\cup T_1$.  In addition we note that $V^1+W^1=J_n|_{T_0\cup T_1}$ is the restriction of the optimal current and is hence a current on $T_0\cup T_1$.  Combining these with the definition~\eqref{eqn:VjandWj} we see that each of $V^l-V^{l+1}$, $W^l+W^{l+1}$ and $V^l+W^{l+1}$ are currents on $T_l\cup T_{l+1}$ as they are obtained from the $l=0$ case by rotations.

The divergence $\nabla\cdot J$ then vanishes on the common boundary of $T_l$ and $T_{l+1}$ by writing $J|_{T_l\cup T_{l+1}}$ as the linear combination $-\alpha_{l+1}(V^l-V^{l+1})+\beta_l(W^l+W^{l+1})+(\alpha_l+\alpha_{l+1})(V^l+W^{l+1})$.
\end{proof}

We will need currents with specified non-zero fluxes on the three sides at which cells join and zero flux on the other sides. The relevant sides were determined in Section~\ref{ssec:graphs}; they are those which contain a vertex of $G_0$.

\begin{prop}\label{prop:Fncurr}
If $I_j$, $j=0,N,3N-1$ satisfy $\sum_{0,N,3N-1} I_j=0$ then there is a current $J$ on $F_n$ with flux $I_j$ on $L_j\cap F_n$ for $j=0,N,3N-1$ and zero on all other $L_j\cap F_n$ and that has energy
\begin{equation*}
	E_{F_n}(J,J)
	\leq \Bigl(\frac{N^2}4 E(V) + \frac{N^2}{18}(11N-8) E(W)\Bigr)\sum_{0,N,3N-1} I_j^2
	\leq \frac{11}9 N^2 R_n \sum_{0,N,3N-1} I_j^2
	\end{equation*}
\end{prop}

\begin{proof}
Write $\Lambda'(N)=\{0,N,3N-1\}$.  Define coefficients $\beta_j$ by
\begin{align*}
	\beta_j= \begin{cases}
		I_N-I_{3N-1} &\text{ if $j=0$}\\
		I_{3N-1}-I_0 &\text{ if $j=N$}\\
		I_0 - I_N &\text{ if $j=3N-1$}
		\end{cases}
	&& 
	\beta_j= \begin{cases}
		2I_N-2I_0 &\text{ if $1\leq j\leq N-1$}\\
		2I_{3N-1}-2I_N &\text{ if $N+1\leq j\leq 3N-2$}\\
		2I_0 - 2I_{3N-1} &\text{ if $3N\leq j\leq 4N-1$}
		\end{cases}
	\end{align*}
and let
\begin{equation*}
	J= -\frac N2 \sum_{j\in\Lambda'} I_j V^j + \frac N6\sum_{j\in\Lambda} \beta_jW^j.
	\end{equation*}
Then $J$ is of the form $\sum_j \alpha_j V^j+\beta_jW^j$ with $\alpha_j=-\frac N2 I_j$ for $j\in\Lambda'$ and zero otherwise.  One can verify the conditions of Lemma~\ref{lem:gluingVW}, so $J$ is a current.  Moreover, all of the $W^j$ have zero flux through $L_j\cap F_n$, and $V^j$ has flux $-\frac2N$ through $L_j\cap F_n$, thus the flux of $J$ is as stated.

By the orthogonality in Lemma~\ref{lem:VorthogW} and~\eqref{eqn:EVEW},
\begin{equation*}
	E_{F_n}(J,J)
	=\frac{N^2}4 E(V)\sum_{\Lambda'} I_j^2 + \frac{N^2}{36} E(W)\sum_{\Lambda} \beta_j^2.
	\end{equation*}
It is straightforward to compute
\begin{align*}
	\sum_{\Lambda} \beta_j^2
	&=(4N-3)(I_N-I_0)^2 + (8N-7)(I_{3N-1}-I_N)^2 + (4N+1)(I_0-I_{3N-1})^2\\
	&= (16N-9)I_0^2+ (20N-17)I_N^2+(20N-13)I_{3N-1}^2 + 4(N-1)I_0I_N + 4(N-2)I_0I_{3N-1}
	\end{align*}
where we used $-2I_NI_{3N-1}=I_0^2+I_N^2+I_{3N-1}^2+2I_0I_N+2I_0I_{3N-1}$, which was obtained by squaring $\sum_{\Lambda'}I_j=0$. Then the bound $2I_0I_j\leq \frac32I_0^2+\frac23I_j^2$ for $j=N,3N-1$ gives, also using $N\geq2$,
\begin{equation*}
	\sum_{\Lambda} \beta_j^2
	\leq (22N-18)I_0^2 + (22N-19)I_N^2 + (22N-16)I_{3N-1}^2.
	\end{equation*}
This gives the energy estimate. The bound by $R_n$ is from the expression following~\eqref{eqn:EVEW}.
\end{proof}

Having established these results on currents, we turn to considering potentials, which will be built from the functions $v^j$ and $w^j$ so as to have specified boundary data at those $C_j$ which are vertices of the graph $D_0$ from Section~\ref{ssec:graphs}.

\begin{lemma}\label{lem:ctspot}
The function which is $v^1+w^1+v^0-w^0$ on $T_0 \cup T_1$ and zero on  $F_n \setminus T_0 \cup T_1$ defines a potential on $F_n$.
\end{lemma}
\begin{proof}
Recall from Section~\ref{ssec:Lipdom} that a function is a potential if it is in $H^2$ on the domain.  Since the given function is the restriction of a harmonic function as in Theorem~\ref{thm:BVPsoln} to both $T_0$ and $T_1$, and is zero on the complement of these, it is in $H^1$ on each of these sets separately.  Moreover, these sets meet along the intersection of $F_n$ with three distinct lines (corresponding to the common boundary of $T_j$ and $T_{j+1}$ for $j=-1,0,1$), so Lemma~\ref{lem:potandcurglue}(i) is applicable in each case and we see the function is in $H^1(F_n)$ with $L^2$ boundary values if and only if the pieces agree a.e.\ on the common boundary. In what follows we suppress the ``a.e.'' to avoid repetition.

The proof that the pieces match uses the symmetries $v^0(\bar{z})=v^0(z)$ and $w^0(\bar{z})=-w^0(z)$, which follow from Lemma~\ref{lem:Jconj} in the same manner as the proofs of~\eqref{eqn:V0sym} and~\eqref{eqn:W0antisym}.   Note that $z\mapsto\theta(\bar{z})$ is an isometry of $T_0\cup T_1$ and compute from the symmetries and~\eqref{eqn:vjandwj} that
\begin{gather}
	v^0(\theta(\bar{z})) =v^0(\theta^{-1}(z))=v^1(z), \notag\\
	w^1(\theta(\bar{z}))= w^0(\bar{z}) = -w^0(z). \label{eqn:w1w0antisym}
	\end{gather}

Now observe that $v^0+w^1$ is in $H^1(T_0\cup T_1)$ because it is the restriction of the optimal potential $u_n$ to this set, see~\eqref{eqn:vjandwj}, and this latter is a harmonic function as in Theorem~\ref{thm:BVPsoln}.  Using the preceding it follows that $(v^1-w^0)(z)=(v^0+w^1)(\theta(\bar{z}))$ is in $H^1(T_0\cup T_1)$, and therefore so is $v^1+w^1+v^0-w^0$.

What is more, if $z$ is in the common boundary of $T_0$ and $T_1$ then $\theta(\bar{z})=z$ and thus~\eqref{eqn:w1w0antisym} gives $w^1(z)=-w^0(z)$.  Then $v^0+w^1\in H^1$ at such points implies $v^0(z)=-w^0(z)$, but this says $v^0+w^0$ vanishes  on the common boundary of $T_0$ and $T_1$. Thus $v^1+w^1=v^0+w^0\circ\theta^{-1}$ vanishes on the common boundary of $T_1$ and $T_2$, and $(v^0-w^0)(z)=(v^0+w^0)(\bar{z})$ vanishes on the common boundary of $T_0$ and $T_{-1}$.  Together these show $v^1+w^1+v^0-w^0$  vanishes on the boundary of $T_0\cup T_1$ in $F_n$, so the zero extension to $F_n\setminus(T_0\cup T_1)$ is in $H^2$ and the proof is complete.
\end{proof}

For the following proposition we recall that the harmonic functions from Theorem~\ref{thm:BVPsoln} are continuous on the sets $A_n$ and $B_n$ in $F_n$, thus we may refer to their values at the points $C_j$.

\begin{prop}\label{prop:Fnpot}
Given a function $u$ on $D_0$ that is harmonic at $0$ there is a potential $f$ on $F_n$ having a representative with $f(C_j)=u(C_j)$ for $C_j\in D_0$ and
\begin{equation*}
	\DF_{F_n}(f,f)
		\leq  \bigl(\DF_n(v)+\DF_n(w)\bigr)\sum_{C_j\in D_0} \bigl( u(C_j)-u(0)\bigr)^2
		=\frac2N R_n^{-1} \sum_{C_j\in D_0} \bigl( u(C_j)-u(0)\bigr)^2.
		\end{equation*}
If $f(C_{j+1})=f(C_j)$ for some $j\in\{0,N,3N-1\}$ then $f$ is constant on the edge $L_j\cap F_n$.
\end{prop}
\begin{proof}
Let  $z_j=u(C_j)-u(0)$ for $C_j\in D_0$ and $z_j=0$ otherwise.  With indices modulo $4N$, define
\begin{equation*}
	f= u(0)+\frac12 \sum_{j\in\Lambda(N)} z_j( w^{j-1} - v^{j-1}-  v^j -w^j)
	=u(0)+ \frac12 \sum_j (z_{j+1}-z_j)w^j -  (z_{j+1}+z_j) v^j.
	\end{equation*}
This is a linear combination of rotations of the function in Lemma~\ref{lem:ctspot}, so it is a potential.  Using $w^{j-1}(C_j)=1$, $v^{j-1}(C_j)=w^{j-1}(C_j)=w^j(C_j)=-1$ and  $v^l(C_j)=w^l(C_j)=0$ for $l\neq j,j+1$ we easily see $f(C_j)=u(C_j)$ for $C_j\in D_0$.

From the orthogonality in Lemma~\ref{lem:VorthogW} and~\eqref{eqn:EvEw} we have
\begin{equation*}
	\DF_{F_n}(f,f)
	=\frac14 \DF_n(v)   \sum_{j\in\Lambda'''} (z_{j+1}+z_j)^2 + \frac14 \DF_n(w) \sum_{j\in\Lambda'''} (z_{j+1}-z_j)^2
	\leq (\DF_n(v)+\DF_n(w)) \sum_j z_j^2
	\end{equation*}
where we  used $(z_{j+1}\pm z_j)^2\leq 2z_{j+1}^2+2z_j^2$.  The remaining part of the asserted energy bound is from from~\eqref{eqn:Rnbyvw}.

Finally, suppose there is $j\in\{0,N,3N-1\}$ for which $f(C_{j+1})=f(C_j)$.  Then $z_{j+1}=z_j$, and on the edge $L_j\cap F_n$ we have  $f=u(0)-(z_{j+1}+z_j)v^j$.  However~\eqref{eqn:vjandwj} says  $v^j$ comes from the restriction of $u_n$ to $L_0$, where $u_n\equiv1$, so $v^j$ is constant on $L_j\cap F_n$ and so is $f$.
\end{proof}

\section{Bounds}

Our main resistance estimate is obtained from the results of the previous sections by constructing a feasible current and potential on $F_{m+n}$.  We use the optimal current on $G_m$ and optimal potential on $D_m$ to define boundary data on $m$-cells that are copies of $F_n$, and then build matching currents and potentials from Propositions~\ref{prop:Fncurr} and~\ref{prop:Fnpot} to prove the following theorem.

\begin{theorem}\label{thm:mainest}
For $n\geq0$ and $m\geq1$
\begin{equation*}
	\frac9{44N} R_0^{-1}  R_n R_m \leq R_{m+n}\leq  \frac{44N}9   R_0^{-1}R_n R_m.
	\end{equation*}
\end{theorem}
\begin{proof}
For fixed $m\geq 1$ let $\tilde{I}_m^G$ be the optimal current on the graph $G_m$ and $\tilde{u}_m^D$ be the optimal potential on the graph $D_m$, both for the sets $A_m$ and $B_m$.  Recall that for each cell we have an address $w=w_1\dotsm w_m$ and a map $\psi_w$ as in Section~\ref{ssec:graphs} so that $\tilde{I}_m^G\circ\psi_w$ is a current on $G_0$ and $\tilde{u}_m^D\circ\psi_w$ is a potential on $D_0$.

Now fix $n\geq 0$ and consider $F_{m+n}$. Then $\psi_w$ maps $F_n$ to the $m$-cell of $F_{m+n}$ with address $w$, and we write $J_w$ for the current from Proposition~\ref{prop:Fncurr} with fluxes from $\tilde{I}_m^G\circ\psi_w$ and $f_w$ for the potential from Proposition~\ref{prop:Fnpot} with boundary data from $\tilde{u}_m^D\circ\psi_w$.  In particular, summing over all words of length $m$ we have from these propositions and the optimality of the current and potential that
\begin{gather}
	E_{F_{m+n}}\Bigl(\sum_w J_w,\sum_w J_w\Bigr)
	= \sum_w E_{F_n}(J_w,J_w)
	\leq \frac{11}9 N^2 R_n E_{G_m}(\tilde{I}_m^G,\tilde{I}_m^G) 
	=  \frac{11}9 N^2 R_n R_m^G \label{eqn:currentbd}\\
	\DF_{F_{m+n}}\Bigl(\sum_w f_w,\sum_w f_w\Bigr)
	= \sum_w \DF_{F_n}(f_w,f_w)
	\leq \frac2N R_n^{-1} \DF_{D_m} (\tilde{u}_m^D,\tilde{u}_m^D)
	=  \frac2N R_n^{-1} (R_m^D)^{-1}. \label{eqn:potentialbd}
	\end{gather}
	
Since $\tilde{I}_m^G$ is a current, its flux through the edges incident at a non-boundary point is zero. Using this fact at the vertex on the center of a side where two $m$-cells meet we see that the net flux of $\sum_w J_w$ through such a side is zero.
What this means for $\sum_w J_w$ is that the currents in the cells that meet on this side are weighted to have equal and opposite flux through the side.  However, examining the construction of $J_w$ it is apparent that the term providing the flux through this side is a  (scaled) copy of $V^j$ from~\eqref{eqn:VjandWj}. Using that $V^j$ is a rotate of $V^0$ and that $V^0(\bar{z})=V^0(z)$ from~\eqref{eqn:V0sym}, we see that all terms in $\sum_w J_w$ that provide flux through the sides where $m$-cells meet are
 multiples of a single vector field.  It follows that the cancellation of the net flux guarantees cancellation of the fields in the sense of Lemma~\ref{lem:potandcurglue}(ii).  Thus we conclude that $\sum_w J_w$ is a current on $F_{m+n}$.  Its net flux through a boundary edge is the same as that of $\tilde{I}_m^G$, so is $-1$ through $A_{m+n}$ and $1$ through $B_{m+n}$.  Hence $\sum_w J_w$ is a feasible current from $A_{m+n}$ to $B_{m+n}$ on $F_{m+n}$, and~\eqref{eqn:currentbd} together with Theorem~\ref{thm:resistfromcurrent} implies
\begin{equation}\label{eqn:Rupper}
	R_{m+n} \leq E_{F_{m+n}}\Bigl(\sum_w J_w,\sum_w J_w\Bigr) \leq  \frac{11}9 N^2 R_n R_m^G.
	\end{equation}

Similarly, we can see that $\sum_w f_w$ is a potential on $F_{m+n}$. Each side where two $m$-cells meet  is the line segment at the intersection of the closures of copies of sectors $T_j$ and $T_{j'}$ under  maps $\psi_w$, $\psi_{w'}$ corresponding to the $m$-cells. We see that $\sum_w f_w$ coincides with $\tilde{u}_m^D$ at the endpoints of this line segment, while along the line it is a linear combination of $v^j$ and $w^j$ as in Proposition~\ref{prop:Fnpot}.  This linear combination depends only on the endpoint values, so is the same on the line from $\psi_w(T_j)$ as on the line from $\psi_{w'}(T_{j'})$. Hence  $\sum_w f_w$ is a potential on $F_{m+n}$.  Since $\tilde{u}_m^D$ is $0$ at all endpoints of sides of cells in $A_{m+n}$ and $1$ at all endpoints of sides in $B_{m+n}$, the final result of Proposition~\ref{prop:Fnpot} ensures $\sum_w f_w$ has value $0$ on $A_{m+n}$ and $1$ on $B_{m+n}$, so is a feasible potential.  Combining this with~\eqref{eqn:potentialbd} and~\eqref{eq:domainresist} gives
\begin{equation}\label{eqn:Rlower}
	R_{m+n}^{-1} \leq \DF_{F_{m+n}}\Bigl(\sum_w f_w,\sum_w f_w\Bigr) \leq \frac2N R_n^{-1} (R_m^D)^{-1}.
	\end{equation}

Our estimates~\eqref{eqn:Rupper} and~\eqref{eqn:Rlower}, together with Lemma~\ref{lem:RGRD},  give for $n\geq 0$, $m\geq1$ that
\begin{equation*}
	\frac N2 R_n R_m^D \leq R_{m+n} \leq \frac{11}9 N^2 R_n R_m^G = \frac{22}9 N^2 R_n R_m^D.
	\end{equation*}
In particular, for $n=0$ we have $R_m^D\leq \frac2N R_0^{-1} R_m$  and $\frac9{22N^2} R_0^{-1}R_m\leq R_m^D$, which may be substituted into the previous expression to obtain the theorem.
\end{proof}

\bibliographystyle{plain}
\bibliography{ms}

\begin{thebibliography}{10}

\bibitem{Andrews}
Ulysses A.~IV Andrews.
\newblock {\em Existence of Diffusions on 4N Carpets}.
\newblock PhD thesis, University of Connecticut, 2017.
\newblock https://opencommons.uconn.edu/dissertations/1477.

\bibitem{BB1}
M.~T. Barlow and R.~F. Bass.
\newblock On the resistance of the {S}ierpi\'{n}ski carpet.
\newblock {\em Proc. Roy. Soc. London Ser. A}, 431(1882):345--360, 1990.

\bibitem{BBS}
M.~T. Barlow, R.~F. Bass, and J.~D. Sherwood.
\newblock Resistance and spectral dimension of {S}ierpi\'{n}ski carpets.
\newblock {\em J. Phys. A}, 23(6):L253--L258, 1990.

\bibitem{BarlowDiffusionsonFractals}
Martin~T. Barlow.
\newblock Diffusions on fractals.
\newblock In {\em Lectures on probability theory and statistics
  ({S}aint-{F}lour, 1995)}, volume 1690 of {\em Lecture Notes in Math.}, pages
  1--121. Springer, Berlin, 1998.

\bibitem{BBExistence}
Martin~T. Barlow and Richard~F. Bass.
\newblock The construction of {B}rownian motion on the {S}ierpi\'{n}ski carpet.
\newblock {\em Ann. Inst. H. Poincar\'{e} Probab. Statist.}, 25(3):225--257,
  1989.

\bibitem{BBTransDensities}
Martin~T. Barlow and Richard~F. Bass.
\newblock Transition densities for {B}rownian motion on the {S}ierpi\'{n}ski
  carpet.
\newblock {\em Probab. Theory Related Fields}, 91(3-4):307--330, 1992.

\bibitem{BB99}
Martin~T. Barlow and Richard~F. Bass.
\newblock Brownian motion and harmonic analysis on {S}ierpinski carpets.
\newblock {\em Canad. J. Math.}, 51(4):673--744, 1999.

\bibitem{BBKT}
Martin~T. Barlow, Richard~F. Bass, Takashi Kumagai, and Alexander Teplyaev.
\newblock Uniqueness of {B}rownian motion on {S}ierpi\'{n}ski carpets.
\newblock {\em J. Eur. Math. Soc. (JEMS)}, 12(3):655--701, 2010.

\bibitem{StrichartzOuterApprox}
Tyrus Berry, Steven~M. Heilman, and Robert~S. Strichartz.
\newblock Outer approximation of the spectrum of a fractal {L}aplacian.
\newblock {\em Experiment. Math.}, 18(4):449--480, 2009.

\bibitem{RBrown}
Russell Brown.
\newblock The mixed problem for {L}aplace's equation in a class of {L}ipschitz
  domains.
\newblock {\em Comm. Partial Differential Equations}, 19(7-8):1217--1233, 1994.

\bibitem{DoyleSnell}
Peter~G. Doyle and J.~Laurie Snell.
\newblock {\em Random walks and electric networks}, volume~22 of {\em Carus
  Mathematical Monographs}.
\newblock Mathematical Association of America, Washington, DC, 1984.

\bibitem{EvansGariepy}
Lawrence~C. Evans and Ronald~F. Gariepy.
\newblock {\em Measure theory and fine properties of functions}.
\newblock Studies in Advanced Mathematics. CRC Press, Boca Raton, FL, 1992.

\bibitem{GrigYang}
Alexander Grigor'yan and Meng Yang.
\newblock Local and non-local {D}irichlet forms on the {S}ierpi\'{n}ski carpet.
\newblock {\em Trans. Amer. Math. Soc.}, 372(6):3985--4030, 2019.

\bibitem{Grisvard}
Pierre Grisvard.
\newblock {\em Elliptic problems in nonsmooth domains}, volume~69 of {\em
  Classics in Applied Mathematics}.
\newblock Society for Industrial and Applied Mathematics (SIAM), Philadelphia,
  PA, 2011.

\bibitem{Hutch81}
John~E. Hutchinson.
\newblock Fractals and self-similarity.
\newblock {\em Indiana Univ. Math. J.}, 30(5):713--747, 1981.

\bibitem{JonssonWallin}
Alf Jonsson and Hans Wallin.
\newblock Function spaces on subsets of {${\bf R}^n$}.
\newblock {\em Math. Rep.}, 2(1):xiv+221, 1984.

\bibitem{KajinoSA}
Naotaka Kajino.
\newblock Spectral asymptotics for {L}aplacians on self-similar sets.
\newblock {\em J. Funct. Anal.}, 258(4):1310--1360, 2010.

\bibitem{KajinoElementaryWD}
Naotaka Kajino.
\newblock An elementary proof of walk dimension being greater than two for
  {B}rownian motion on {S}ierpi\'{n}ski carpets.
\newblock arXiv:2005.02524, 2020.

\bibitem{KPBT19}
Daniel~J. Kelleher, Hugo Panzo, Antoni Brzoska, and Alexander Teplyaev.
\newblock Dual graphs and modified {B}arlow-{B}ass resistance estimates for
  repeated barycentric subdivisions.
\newblock {\em Discrete Contin. Dyn. Syst. Ser. S}, 12(1):27--42, 2019.

\bibitem{Kigami}
Jun Kigami.
\newblock {\em Analysis on Fractals}, volume 143 of {\em Cambridge Tracts in
  Mathematics}.
\newblock Cambridge University Press, Cambridge, 2001.

\bibitem{McGillivray}
I.~McGillivray.
\newblock Resistance in higher-dimensional {S}ierpi\'{n}ski carpets.
\newblock {\em Potential Anal.}, 16(3):289--303, 2002.

\bibitem{StrichartzPeanoCurve}
Denali Molitor, Nadia Ott, and Robert Strichartz.
\newblock Using {P}eano curves to construct {L}aplacians on fractals.
\newblock {\em Fractals}, 23(4):1550048, 29, 2015.

\bibitem{Steindiffprops}
Elias~M. Stein.
\newblock {\em Singular integrals and differentiability properties of
  functions}.
\newblock Princeton Mathematical Series, No. 30. Princeton University Press,
  Princeton, N.J., 1970.

\bibitem{Ziemer}
William~P. Ziemer.
\newblock {\em Weakly differentiable functions}, volume 120 of {\em Graduate
  Texts in Mathematics}.
\newblock Springer-Verlag, New York, 1989.
\newblock Sobolev spaces and functions of bounded variation.

\end{thebibliography}

\end{document}